\documentclass{aims}
\usepackage{amsmath}
\usepackage{paralist}
\usepackage{subfig}
\usepackage[dvips]{graphics}
\usepackage{graphics}
\usepackage{epsfig}
\usepackage{graphicx}
\usepackage{epstopdf}
 \usepackage[colorlinks=true]{hyperref}
 \usepackage{graphicx}
\usepackage{color}
\usepackage{ifpdf}
\hypersetup{urlcolor=blue, citecolor=red}

\def\currentvolume{X}
 \def\currentissue{X}
  \def\currentyear{200X}
   \def\currentmonth{XX}
    \def\ppages{X--XX}
     \def\DOI{10.3934/xx.xx.xx.xx}

\def\theequation {\thesection.\arabic{equation}}
\def\thetable{\thesection.\arabic{table}}
\def\thefigure{\thesection.\arabic{figure}}

\theoremstyle{plain}
\newtheorem{theorem}{Theorem}[section]
\newtheorem{lemma}{Lemma}[section]

\newtheorem{proposition}{Proposition}[section]

\theoremstyle{definition}
\newtheorem{definition}{Definition}[section]
\newtheorem{remark}{Remark}[section]

\newcommand{\bx}{{\mathbf x}}
\newcommand{\bX}{{\mathbf X}}
\newcommand{\nn}{\nonumber}
\newcommand{\be} {\begin{equation}}
\newcommand{\ee}{\end{equation}}
\newcommand{\ba}{\begin{array} }
\newcommand{\ea}{\end{array}}
\newcommand{\bea}{\begin{eqnarray}}
\newcommand{\eea}{\end{eqnarray} }
\newcommand{\beas}{\begin{eqnarray*}}
\newcommand{\eeas}{\end{eqnarray*} }
\newcommand{\bnn}{{\bf n}}
\newcommand{\bpro}{\begin{proposition} }
\newcommand{\epro}{\end{proposition} }

\newcommand{\Tmax}{T_{\rm max}}

\title[Quantized vortex dynamics in superconductivity]
{Quantized vortex dynamics and interaction patterns in superconductivity based on the reduced dynamical law}

\author[Zhiguo Xu, Weizhu Bao and Shaoyun Shi]{}


 \email{xuzg2014@jlu.edu.cn}
 \email{matbaowz@nus.edu.sg, http://www.math.nus.edu.sg/\~{ }bao}
 \email{shisy@jlu.edu.cn}

\begin{document}
\maketitle

\centerline{\scshape Zhiguo Xu}
\smallskip
{\footnotesize
 \centerline{School of Mathematics, Jilin University, Changchun 130012, P.
R. China}
}

\bigskip

\centerline{\scshape Weizhu Bao}
\smallskip
{\footnotesize
 \centerline{Department of Mathematics,  National University of Singapore, Singapore 119076}
}

\bigskip

\centerline{\scshape Shaoyun Shi}
\smallskip
{\footnotesize
 \centerline{School of Mathematics, Jilin University, Changchun 130012, P.
R. China}
}

\begin{abstract}
We study analytically and numerically stability and interaction patterns
of quantized vortex lattices governed by the reduced dynamical law -- a system of ordinary differential equations (ODEs) -- in superconductivity.
By deriving several non-autonomous first integrals of the ODEs, we obtain qualitatively dynamical properties of a cluster of quantized vortices, including global existence, finite time collision, equilibrium solution
and invariant solution manifolds. For a vortex lattice with 3 vortices, we establish orbital stability when they have the same winding number and find different collision patterns when they have different winding numbers. In addition, under several special initial setups,
we can obtain analytical solutions for the nonlinear ODEs.
\end{abstract}

\section{Introduction}
\setcounter{equation}{0}
\setcounter{figure}{0}
In this paper, we study analytically and numerically stability and interaction patterns of the following system of
ordinary differential equations (ODEs) describing the dynamics
of $N\ge2$ quantized vortices in superconductivity
based on the reduced dynamical law \cite{Neu2,Jerrard,Lin96,Zhang1,Zhang2}
\be\label{GLE}
\dot{\mathbf{x}}_j(t)=2m_j\displaystyle\sum_{k=1,k\neq
j}^{N}m_k\frac{\mathbf{x}_j(t)-\mathbf{x}_k(t)}{\left|
\mathbf{x}_j(t)-\mathbf{x}_k(t)\right|^2},\qquad 1\le j\le N, \quad t> 0,
\ee
with initial data
\be\label{GLEi}
\mathbf{x}_j(0)=\mathbf x_j^0=(x_j^0,y_j^0)^T\in {\mathbb R}^2,\quad
1\le j\le N.
\ee
Here $t$ is time, $\mathbf x_j(t)=(x_j(t), y_j(t))^T\in \mathbb R^2$
is the center of the $j$-th ($1\le j\le N$) quantized vortex at time $t$,
$m_j=+1$ or $-1$ is the winding number or index or circulation of
the $j$-th ($1\le j\le N$) quantized vortex. We always assume that the initial data satisfies $\bX^0:=(\bx_1^0,\ldots,\bx_N^0)\in
{\mathbb R}_*^{2\times N}:=
\{\bX=(\bx_1,\ldots,\bx_N)\in {\mathbb R}^{2\times N}\ |\
\bx_j\ne \bx_l\in {\mathbb R}^2\ \hbox{for}\ 1\le j<l\le N\}$
and denote its mass center as $\bar\bx^0:=\frac{1}{N}\sum_{j=1}^N\bx_j^0$.
Throughout this paper, we assume that $N\ge2$.

The ODEs \eqref{GLE} with \eqref{GLEi} was derived asymptotically
as a reduced dynamical law
for the dynamics of $N$ quantized vortices -- particle-like
or topological defects -- in the
Ginzburg-Landau equation \cite{Neu1,Jerrard,Lin96}
\be\label{pgl}
\partial_t \psi(\mathbf x,t)
=\nabla^2\psi(\mathbf x,
t)+\frac{1}{\varepsilon^2}(1-|\psi(\mathbf x,
t)|^2)\psi(\mathbf x,t), \quad \mathbf x\in\mathbb{R}^2,\quad t>0,
\ee
with initial condition
\be\label{pgli}
\psi(\mathbf x,0)=\psi_0(\mathbf x)=\Pi_{j=1}^N \phi_{m_j}(\bx-\bx_j^0), \qquad \mathbf
x\in\mathbb{R}^2,
\ee
for superconductivity when either $\varepsilon=1$ and $d_{\rm min}^0:=\min_{1\le j<l\le N}|\bx_j^0-\bx_l^0|\to \infty$ \cite{Neu2} or
for a given $\bX^0\in {\mathbb R}_*^{2\times N}$ and $\varepsilon\to 0^+$ \cite{Jerrard,Lin96}.
Here $\mathbf x=(x,y)^T\in\mathbb R^2$ is the Cartesian
coordinates in two dimensions (2D), $\psi:=\psi(\mathbf x, t)$ is a complex-valued
order parameter, $\varepsilon>0$ is a constant, and $\phi_{m}(\bx)=f(r)e^{im\theta}$ ($m=+1$ or $-1$) with $(r,\theta)$ the polar coordinates
in 2D and $f(r)$ satisfying \cite{Neu2,Jerrard,Lin96,Zhang1,Zhang2}
\beas
&&\frac{1}{r}\frac{d}{dr}\left(r\frac{df(r)}{dr}\right)-\frac{1}{r^2}f(r)
+\frac{1}{\varepsilon^2}(1-f^2(r))f(r)=0, \qquad 0<r<+\infty,\\
&&f(0)=0, \qquad \lim_{r\to +\infty} f(r)=1.
\eeas
Here $\phi_{m}(\bx)$ is a typical quantized vortex in 2D,
which is zero of the order parameter at the
vortex center located at the origin and has localized phase
singularity with integer $m$ topological charge usually
called also as winding number or index or circulation.
In fact, quantized
vortices have been widely observed in superconductor \cite{E,Lin96,Bao2},
liquid helium \cite{Newton}, Bose-Einstein condensates \cite{Pit,Bao0,Klein};
and they are key signatures of superconductivity and superfluidity.
The study of quantized vortices and their dynamics is one of the
most important and fundamental problems in superconductivity and
superfluidity \cite{Neu2,Bao1,Mironescu,Sandier,Bauman,Chapman,Colliander,Du,
Lange,Lin98,Bao3,Ovchinnikov1,Ovchinnikov2}.

Based on the reduced dynamical law, i.e. \eqref{GLE}, for
the quantized vortex dynamics in superconductivity, when two quantized
vortices have the same winding number (i.e. vortex pair), they undergo a
repulsive interaction; and respectively, when they have opposite winding numbers (i.e. vortex dipole or vortex-antivortex), they undergo an attractive interaction \cite{Neu2,Zhang1,Zhang2}. For $N\ge2$ and $\bX^0\in {\mathbb R}_*^{2\times N}$, it is straightforward to obtain local existence
of the ODEs \eqref{GLE} with \eqref{GLEi} by the standard theory of ODEs.
Specifically, when $N=2$, one can obtain explicitly the
analytical solution of \eqref{GLE} with \eqref{GLEi}: when $m_1=m_2$ (i.e. vortex pair),
the two vortices move outwards by repelling each other
along the line passing through their initial locations $\bx_1^0\ne \bx_2^0$ and they never collide at finite time; and when $m_1=-m_2$ (i.e. vortex dipole or vortex-antivortex), the two vortices move towards each other
along the line passing through their initial locations $\bx_1^0\ne \bx_2^0$ and they will collide at $\frac{1}{2}\left(\bx_1^0+\bx_2^0\right)$ in finite time \cite{Neu2,Zhang1,Zhang2}. For analytical solutions of the ODEs
\eqref{GLE} with several special initial setups in \eqref{GLEi},
we refer to \cite{Zhang1,Zhang2} and references therein.
In addition, define the {\sl mass center} of the $N$ vortices as
\be
\bar{\bx}(t):=\frac{1}{N}\sum_{j=1}^N \bx_j(t), \qquad t\ge0,
\ee
then it was proven that the mass center is conserved under the dynamics
of \eqref{GLE} with \eqref{GLEi} \cite{Zhang1,Zhang2}
\be \label{mccon1}
\bar{\bx}(t)\equiv \bar{\bx}(0)
=\bar{\bx}^0,
\qquad t\ge0.
\ee

Introduce
\be
W({\mathbf X})=-\sum_{1\le j\ne k\le N}m_jm_k\, \ln |{\mathbf x}_j-{\mathbf x}_k|=-\ln\, \displaystyle\prod_{1\le j\ne k\le N}|{\mathbf x}_j-{\mathbf x}_k|^{m_jm_k}, \qquad
{\mathbf X}\in {\mathbb R}_*^{2\times N},\ee
then \eqref{GLE} can be reformulated as
\be\label{GLEv}
\dot{\mathbf{X}}(t)=-\nabla_{\mathbf{X}}W({\mathbf X}),
\qquad  t>0,
\ee
which implies that
\be
W({\mathbf X}(t_2))\le W({\mathbf X}(t_1))\le W({\mathbf X}(0))=W({\mathbf X}^0), \qquad 0\le t_1\le t_2.
\ee
In addition, let ${\mathbf z}_j(t):=x_j(t)+iy_j(t)\in {\mathbb C}$ for $1\le j\le N$, then
\eqref{GLE} can be reformulated as
\be\label{GLEc}
\dot{\mathbf{z}}_j(t)=2m_j\displaystyle\sum_{k=1,k\neq
j}^{N}m_k\frac{\mathbf{z}_j(t)-\mathbf{z}_k(t)}{\left|
\mathbf{z}_j(t)-\mathbf{z}_k(t)\right|^2}=2m_j\displaystyle\sum_{k=1,k\neq
j}^{N}\frac{m_k}{\overline{\mathbf{z}}_j(t)-\overline{\mathbf{z}}_k(t)},\quad 1\le j\le N,
\quad t>0,
\ee
where $\bar{z}$ denotes the complex conjugate of $z\in {\mathbb C}$.

For rigorous mathematical justification of the derivation of the above
reduced dynamical law \eqref{GLE} with \eqref{GLEi} for superconductivity,
we refer to \cite{Jerrard,Lin96} and references therein,
and respectively, for numerical comparison of quantized vortex center dynamics under the Ginzburg-Landau equation \eqref{pgl} with
\eqref{pgli} and its corresponding
reduced dynamical law \eqref{GLE} with \eqref{GLEi}, we refer to \cite{Zhang1,Zhang2} and references therein.
Based on the mathematical and numerical results \cite{Jerrard,Lin96,Zhang1,Zhang2},
the dynamics of the $N$ quantized vortex centers under
the reduced dynamical law agrees qualitatively (and quantitatively when
they are well-separated) with that under the Ginzburg-Landau equation.
The main aim of this paper is to study analytically and numerically the dynamics and interaction
patterns of the reduced dynamical law \eqref{GLE} with \eqref{GLEi}, which
will generate important insights about quantized vortex dynamics and interaction patterns in superconductivity and is much simpler
than to solve the Ginzburg-Landau equation \eqref{pgl} with
\eqref{pgli}. We establish global existence of the ODEs
\eqref{GLE} when the $N$ quantized vortices have the same winding number
and possible finite time collision when they have opposite winding numbers.
For $N=3$, we prove orbital stability when they have the same winding number and find different collision patterns when they have different winding numbers. Analytical solutions of the ODEs
\eqref{GLE} are obtained under several initial setups with symmetry.

 The paper is organized as follows.
In section 2, we obtain some invariant solution manifolds
and several non-autonomous
first integrals of the ODEs \eqref{GLE} and establish its global existence
when the $N$ quantized vortices have the same winding number
and possible finite time collision when they have opposite winding numbers.
In section 3, we prove orbital stability when they have the same winding number and find different collision patterns when they have different winding numbers for the dynamics of $N=3$ vortices.
Analytical solutions of the ODEs
\eqref{GLE} are presented under several initial setups with symmetry
in  section 4.  Finally, some conclusions are drawn in section 5.

\section{Dynamical properties of a cluster with $N$ quantized vortices}
\setcounter{equation}{0}
\setcounter{figure}{0}
In this section, we establish dynamical properties of the system of ODEs
(\ref{GLE}) with the initial data (\ref{GLEi}) for describing
the dynamics -- reduced dynamical law -- of a cluster with $N$ quantized vortices in superconductivity.

For any two vortices $\bx_j(t)$ and $\bx_l(t)$ ($1\le j<l\le N$),
if there exists a finite time $0<T_c<+\infty$ such that $d_{jl}(t):=|\bx_j(t)-\bx_l(t)|>0$ for $0\le t< T_c$ and
$d_{jl}(T_c)=0$, then we say that they will be {\sl finite time collision}
or annihilation (cf. Fig. \ref{collion1}a);
otherwise, i.e. $d_{jl}(t)>0$ for $t\ge0$, then we say that they will
not collide. When $N\ge 2$ and let $I\subseteq \{1,2,\ldots,N\}$
be a set with at least $2$ elements, if there exists a finite time $0<T_c<+\infty$ such that $\min_{1\le j< l\le N} d_{jl}(t)>0$ for $0\le t< T_c$, $\lim_{t\to T_c^-} \bx_j(t)=\bx^0\in {\mathbb R}^2$ for $j,k\in I$ with $\bx^0$ a fixed point and $\min_{j\in I} d_{jl}(t)>0$ for $0\le t\le  T_c$ and $l\in J:=\{m\ |\ 1\le m \le N, m\notin I\}$,
then we say that all vortices in the set $I$ will form a (finite time) {\sl collision cluster} among the $N$ vortices (cf. Fig. \ref{collion1}b). Define
\[T_{\rm max}=\sup\left\{t\ge0\ |~\mathbf x_j(t)\neq \mathbf
x_l(t),\ \hbox{for all} \ 1\le j\neq l\le N \right\},\]
it is easy to see that $0<T_{\rm max}\le +\infty$ by noting \eqref{GLEi}.
If $T_{\rm max}<+\infty$, a finite time collision happens among
at least two vortices in the $N$ vortices (or the ODEs \eqref{GLE} with \eqref{GLEi} will blow-up at finite time); otherwise, i.e. $T_{\rm max}=+\infty$, there is no collision among all
the $N$ quantized vortices (or the ODEs \eqref{GLE} with \eqref{GLEi}
is global well-posed in time).

\begin{figure}[htbp]
\centerline{\psfig{figure=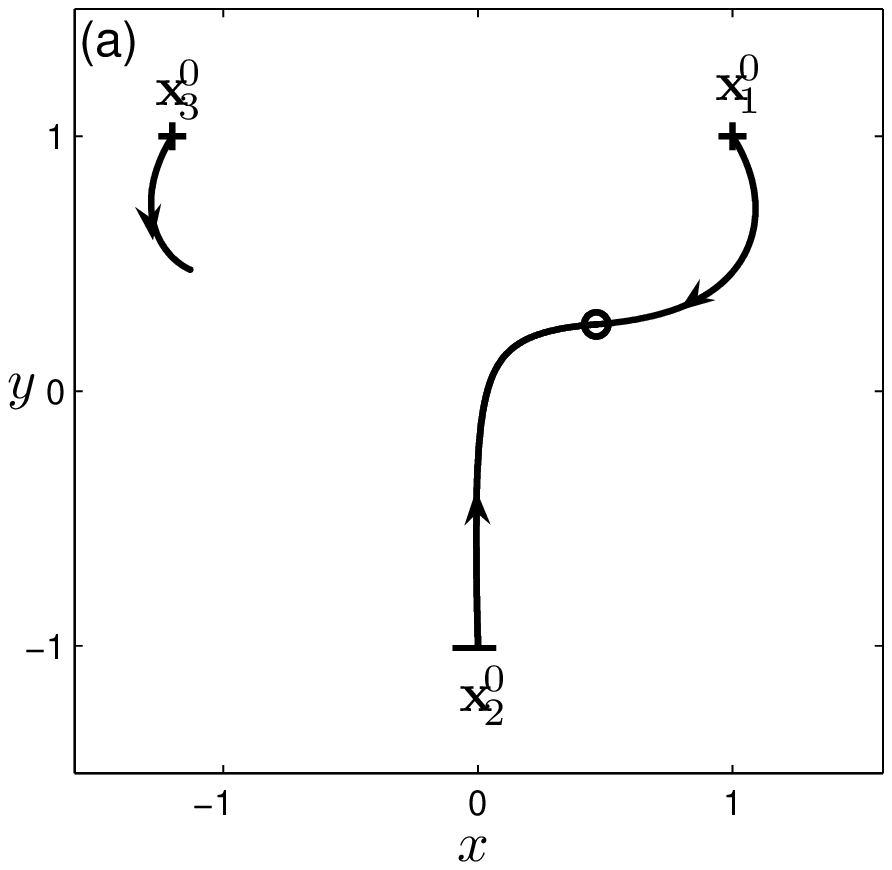,height=6.0cm,width=6.0cm,angle=0}
\psfig{figure=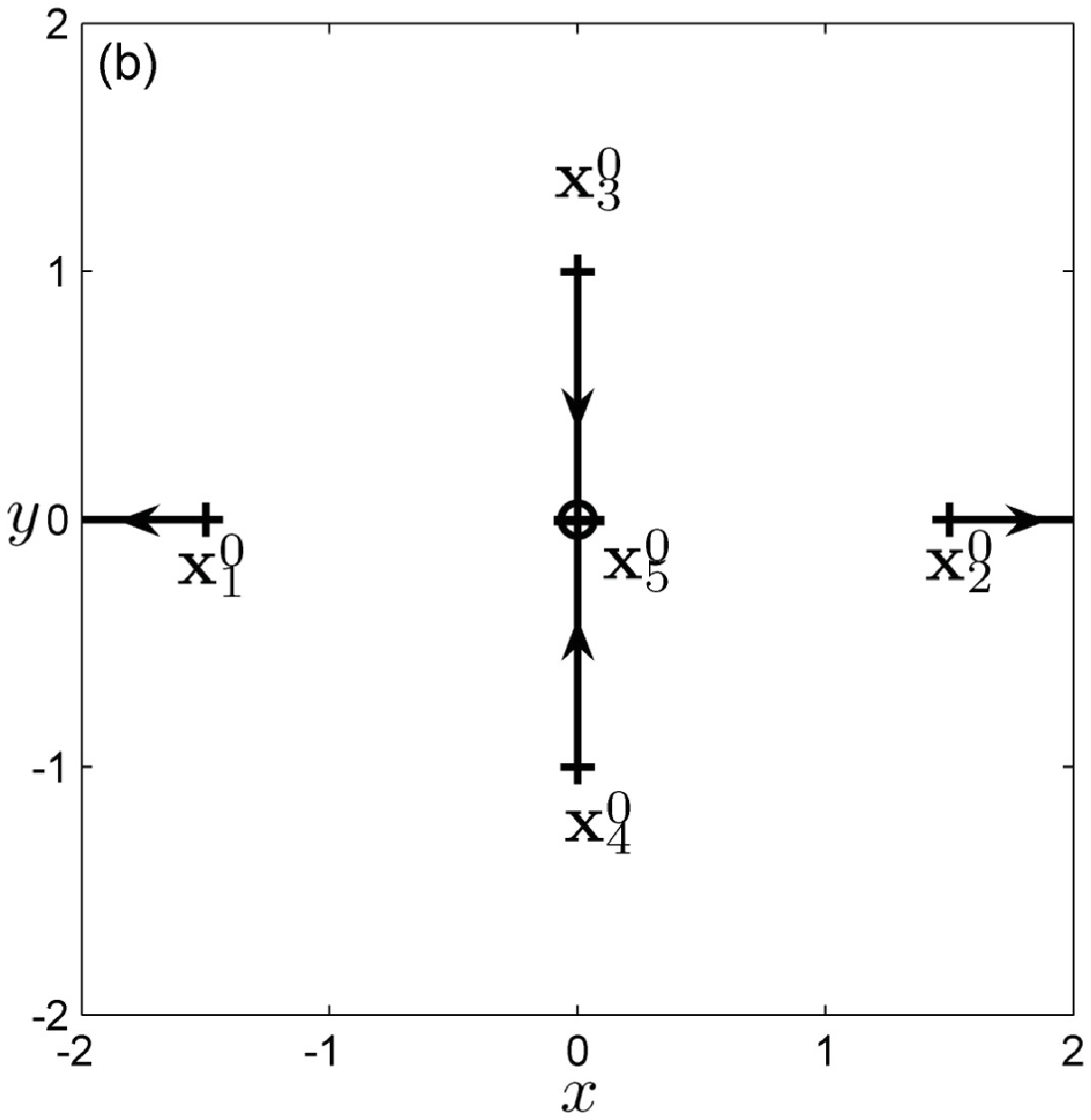,height=5.4cm,width=5.5cm,angle=0}}
\caption{Illustrations of a finite time collision of a vortex dipole in
a vortex cluster with 3 vortices (a) and
a (finite time) collision cluster with 3 vortices in
a vortex cluster with 5 vortices (b). Here and in the following figures,
`+' and `$-$' denote the initial vortex centers with winding numbers
$m=+1$ and $m=-1$, respectively; and `o' denotes the finite time collision position.}
\label{collion1}
\end{figure}

\subsection{Invariant solution manifolds}
Let $\alpha>0$ be a positive constant,
$0\le \theta_0<2\pi$ be a constant,
$\bx^0\in{\mathbb R}^2$ be a given point and $Q(\theta)$ be the rotational matrix defined as
\be
Q(\theta)=\left(\begin{array}{ll}
\cos\theta &-\sin\theta\\
\sin\theta &\ \cos\theta\\
\end{array}\right), \qquad 0\le \theta <2\pi.\nn
\ee

Then it is easy to see that the ODEs \eqref{GLE} with \eqref{GLEi}
is translational and rotational invariant with the proof omitted here for brevity.

\begin{lemma}\label{invar12}
Let $\bX(t)=(\bx_1(t),\bx_2(t),\ldots, \bx_N(t))\in {\mathbb R}_*^{2\times N}$
be the solution of the ODEs \eqref{GLE} with \eqref{GLEi}, then we have

(i) If $\bx_j^0\to \bx_j^0+\bx^0$ for $1\le j\le N$ in \eqref{GLEi},
then $\bx_j(t)\to \bx_j(t)+\bx^0$ for $1\le j\le N$.

(ii) If $\bx_j^0\to \alpha\bx_j^0$ for $1\le j\le N$ in \eqref{GLEi},
then $\bx_j(t)\to \alpha\bx_j(t/\alpha^2)$ for $1\le j\le N$.

(iii) If $\bx_j^0\to Q(\theta)\bx_j^0$ for $1\le j\le N$ in \eqref{GLEi},
then $\bx_j(t)\to Q(\theta)\bx_j(t)$ for $1\le j\le N$.
\end{lemma}

Denote
\[S_{{\bf e}}(\bx^0):=\left\{\bX=(\bx_1,\ldots,\bx_N)\in {\mathbb R}_*^{2\times N} \ |\ |(\bx_j-\bx^0)\cdot{\bf e}|=|\bx_j-\bx^0|\times |{\bf e}|, 1\le j\le N\right\},\]
where ${\bf e}\in {\mathbb R}^2$ is a given unit vector. In fact, $S_{{\bf e}}(\bx^0)$ is a line in 2D passing the point $\bx^0$ and parallel to the unit vector ${\bf e}$. For $\bX=(\bx_1,\ldots,\bx_N)\in {\mathbb R}_*^{2\times N}$, if there exist $\bx^0\in {\mathbb R}^2$ and a unit vector ${\bf e}\in {\mathbb R}^2$ such that $\bX\in S_{{\bf e}}(\bx^0)$, then we say that
$\bX$ is collinear.

\begin{lemma}\label{lmcolin}
If the initial data $\bX^0\in {\mathbb R}_*^{2\times N}$ in \eqref{GLEi} is collinear, i.e.
there exist  $\bx^0\in{\mathbb R}^2$ and a unit vector  ${\bf e}\in {\mathbb R}^2$ such that $\bX^0\in S_{\bf e}(\bx^0)$, then the solution $\bX(t)$ of
\eqref{GLE}-\eqref{GLEi} is collinear, i.e.
$\bX(t)\in S_{\bf e}(\bx^0)$ for $0\le t<T_{\rm max}$.
\end{lemma}

\begin{proof} From $\bX^0\in S_{\bf e}(\bx^0)$, there exist $a_j^0\in {\mathbb R}$ ($1\le j\le N$) satisfying $a_j^0\ne a_l^0$ for $1\le j<l\le N$ such that
\be \label{bxj089}
\bx_j^0=\bx^0+a_j^0 {\bf e}, \qquad 1\le j\le N.
\ee
Noting the symmetric structure in \eqref{GLE} and \eqref{bxj089}, we can assume
\be\label{bxjt678}
\bx_j(t)=\bx^0+a_j(t) {\bf e}, \qquad 1\le j\le N, \qquad t\ge0.
\ee
Plugging \eqref{bxjt678} into \eqref{GLE}, we have
\be\label{GLEaj}
\dot{a}_j(t)=2m_j\sum_{k=1,k\neq
j}^{N}m_k\frac{a_j(t)-a_k(t)}{\left|
a_j(t)-a_k(t)\right|^2},\qquad 1\le j\le N, \quad t> 0,
\ee
with the initial data by noting \eqref{bxj089}
\be\label{aj0345}
a_j(0)=a_j^0, \qquad 1\le j\le N.
\ee
The ODEs \eqref{GLEaj} with \eqref{aj0345} is locally well-posed.
Thus $\bX(t)\in S_{\bf e}(\bx^0)$ for $0\le t<T_{\rm max}$.
\end{proof}

Let ${\bf e}\in {\mathbb R}^2$
be a  unit vector, denote $\theta_N^j:=\frac{2(j-1)\pi}{N}$
and $\bx_j^{(0)}=Q(\theta_N^j+\theta_0){\bf e}$
for $1\le j\le N$ and define
\beas
&&S_{{\bf e}}^N(\bx^0,\theta_0):=\left\{\bX^0_r=(\bx^0+r\bx_1^{(0)},\ldots,\bx^0+
r\bx_N^{(0)})\in {\mathbb R}_*^{2\times N} \ |\ r>0\right\},\\
&&S_{{\bf e}}^N(\bx^0):=\displaystyle\bigcup_{0\le \theta_0<2\pi}S_{{\bf e}}^N(\bx^0,\theta_0).
\eeas

\begin{lemma}\label{plogon1}
Assume the $N$ vortices have the same winding number,
i.e. $m_1=\ldots=m_N=\pm1$. If there exists a unit vector ${\bf e}\in {\mathbb R}^2$, $\bx^0\in{\mathbb R}^2$ and $0\le \theta_0<2\pi$ such that
the initial data  $\bX^0\in
S_{{\bf e}}^N(\bx^0,\theta_0)$ in \eqref{GLEi},  then the solution $\bX(t)$ of \eqref{GLE} satisfies
$\bX(t)\in S_{\bf e}^N(\bx^0,\theta_0)$ for $t\ge0$.
\end{lemma}

\begin{proof} Since $\bX^0\in
S_{{\bf e}}^N(\bx^0,\theta_0)$, there exists a $r_0>0$ such that
\be\label{bxj0654}
\bx_j^0=\bx^0+r_0Q(\theta_N^j+\theta_0){\bf e}, \qquad  1\le j\le N.
\ee
Noting the symmetric structure in \eqref{GLE} and \eqref{bxj0654}, we can assume
\be\label{bxjt6786}
\bx_j(t)=\bx^0+r(t)Q(\theta_N^j+\theta_0){\bf e}, \qquad 1\le j\le N, \qquad t\ge0.
\ee
Plugging \eqref{bxjt6786} into \eqref{GLE}, noting $m_1=\ldots=m_N$ and \eqref{bxj0654}, we have \cite{Zhang1,Zhang2}
\begin{equation*}
\dot{r}(t)=\frac{N-1}{r(t)}, \quad t> 0, \qquad r(0)=r_0,
\end{equation*}
which implies $r(t)=\sqrt{r_0^2+2(N-1)t}$ for $t\ge0$.
Thus $\bX(t)\in S_{\bf e}^N(\bx^0,\theta_0)$ for $t\ge0$.
\end{proof}

From the above two lemmas, for any $\theta_0\in{\mathbb R}$, $\bx^0\in {\mathbb R}^2$ and a unit vector ${\bf e}\in {\mathbb R}^2$, $S_{{\bf e}}(\bx^0)$ is an invariant solution manifold
of the ODEs \eqref{GLE} with \eqref{GLEi}. In addition,
when $m_1=\ldots=m_N$, then  $S_{{\bf e}}^N(\bx^0,\theta_0)$
is also an invariant solution manifold
of the ODEs \eqref{GLE} with \eqref{GLEi}.
Specifically, when  $\bX^0\in S_{{\bf e}}^N({\bf 0},\theta_0)$ and $m_1=\ldots=m_N$, then the ODEs \eqref{GLE} with \eqref{GLEi}
admits the self-similar solution
$\bX(t)=\sqrt{r_0^2+2(N-1)t}\,\bX^0$ with $r_0=|\bx_1^0|=\ldots=|\bx_N^0|$
for $t\ge0$. For more self-similar solutions of the ODEs \eqref{GLE} with special initial setups,
we refer to \cite{Zhang1,Zhang2} and references therein.

\subsection{Non-autonomous first integrals} Let $N^{+}$ and $N^{-}$ be the number of vortices with winding
number $m=1$ and $m=-1$, respectively, then we have
\[0\le N^+\le N,\qquad 0\le N^-\le N, \qquad   N^++N^-=N.\]
In addition, it is easy to get
\bea\label{mjml}
M_0&=&\sum_{1\leq j<l\leq
N}m_jm_l=\frac{1}{2}\sum_{j,l=1,j\ne l}^N m_jm_l
=\frac{N^{+}(N^{+}-1)}{2}+\frac{N^{-}(N^{-}-1)}{2}-N^{+}N^{-}\nn\\
&=&\frac{(N^{+}-N^{-})^2-N}{2}.
\eea

Define 
\bea\label{H12}
H_1(\mathbf{X},t)&=&-4NM_0t+\sum_{1\leq j<l\leq
N}|\mathbf{x}_j-\mathbf{x}_l|^2,\quad
H_2(\mathbf{X},t)=-4M_0t+\sum_{j=1}^{ N}|\mathbf{x}_j|^2, \\
\label{H13}
H_3(\mathbf{X},t)&=&-4(N-2)M_0t+\sum_{1\leq j<l\leq
N}|\mathbf{x}_j+\mathbf{x}_l|^2,\
\bX:=(\bx_1,\bx_2,\ldots, \bx_N)\in {\mathbb R}^{2\times N}, \ t\ge0,\qquad
\eea
then we have

\begin{lemma}\label{lem1}
Let $\bX(t)=(\bx_1(t),\bx_2(t),\ldots, \bx_N(t))\in {\mathbb R}_*^{2\times N}$
be the solution of the ODEs \eqref{GLE} with \eqref{GLEi},  then
$H_1(\bX,t)$, $H_2(\bX,t)$ and $H_3(\bX,t)$
are  non-autonomous first integrals of \eqref{GLE}, i.e.
\bea \label{fi11}
H_1(\mathbf{X}(t), t)&\equiv&H_1^0:=\sum_{1\leq j<l\leq
N}|\mathbf{x}_j^0-\mathbf{x}_l^0|^2,\qquad H_2(\mathbf{X}(t), t)\equiv  H_2^0:=\sum_{j=1}^{N}|\mathbf{x}_j^0|^2, \\
\label{fi12}
H_3(\mathbf{X}(t), t)&\equiv&H_3^0:=\sum_{1\leq j<l\leq
N}|\mathbf{x}_j^0+\mathbf{x}_l^0|^2, \qquad t\ge0.
\eea
Specifically, when $M_0=0$, $H_1(\bX):=H_1(\bX,t)=\sum_{1\le j<l\le N}|\bx_j-\bx_l|^2$ and $H_2(\bX):=H_2(\bX,t)=\sum_{j=1}^N|\bx_j|^2$
are two autonomous first integrals of \eqref{GLE}; and
when either $N=2$ or $M_0=0$, $H_3(\bX):=H_3(\bX,t)=\sum_{1\le j<l\le N}|\bx_j+\bx_l|^2$ is an autonomous first integral of \eqref{GLE}.
\end{lemma}

\begin{proof} Differentiating the left equation in \eqref{H12} (with $\bX=\bX(t)$) with respect to $t$,  we have
\bea \label{H135}
\frac{dH_1(\mathbf{X}(t),t)}{dt}
&=&\left.\nabla_{\mathbf{X}}H_1(\bX,t)
\cdot\dot{\mathbf{X}}+\frac{\partial H_1(\bX,t)}{\partial t}\right|_{\bX=\bX(t)}\nn\\
&=&-4NM_0+2\sum_{1\leq j<l\leq N}(\mathbf
x_j(t)-\mathbf x_l(t))\cdot(\dot{\mathbf x}_j(t)-\dot {\mathbf x}_l(t)),\qquad t\ge0.
\eea
Using summation by parts and noting  \eqref{mjml} and \eqref{GLE}, we obtain
\bea\label{G12}
I&:=&2\sum_{1\leq j<l\leq N}(\mathbf
x_j(t)-\mathbf x_l(t))\cdot(\dot{\mathbf x}_j(t)-\dot {\mathbf x}_l(t))=\sum_{j,l=1,j\ne l}^N(\mathbf
x_j(t)-\mathbf x_l(t))\cdot(\dot{\mathbf x}_j(t)-\dot {\mathbf x}_l(t))\nn\\
&=&2\sum_{j,l=1,j\ne l}^N(\mathbf
x_j(t)-\mathbf x_l(t))\cdot\left[\sum_{k=1,k\ne j}^N m_j m_k \frac{\bx_j(t)-\bx_k(t)} {|\bx_j(t)-\bx_k(t)|^2}-\sum_{k=1,k\ne l}^N m_l m_k \frac{\bx_l(t)-\bx_k(t)} {|\bx_l(t)-\bx_k(t)|^2}\right]\nn\\
&=&2\sum_{j,l=1,j\ne l}^N2m_jm_l\frac{|\bx_j(t)-\bx_l(t)|^2} {|\bx_j(t)-\bx_l(t)|^2}
+2\sum_{j,l=1,j\ne l}^N(\mathbf
x_j(t)-\mathbf x_l(t))\cdot\sum_{k=1,k\ne j,l}^N m_j m_k \frac{\bx_j(t)-\bx_k(t)} {|\bx_j(t)-\bx_k(t)|^2}\nn\\
&&-2\sum_{j,l=1,j\ne l}^N(\mathbf
x_j(t)-\mathbf x_l(t))\cdot\sum_{k=1,k\ne j,l}^N m_lm_k \frac{\bx_l(t)-\bx_k(t)} {|\bx_l(t)-\bx_k(t)|^2}\nn\\
&=&4\sum_{j,l=1,j\ne l}^Nm_jm_l+2\sum_{j,k=1,j\ne k}^N m_j m_k
\sum_{l=1,l\ne j,k}^N  \frac{(\mathbf
x_j(t)-\mathbf x_l(t))\cdot(\bx_j(t)-\bx_k(t))} {|\bx_j(t)-\bx_k(t)|^2}\nn\\
&&-2\sum_{l,k=1,l\ne k}^N m_l m_k
\sum_{j=1,l\ne l,k}^N  \frac{(\mathbf
x_j(t)-\mathbf x_l(t))\cdot(\bx_l(t)-\bx_k(t))} {|\bx_l(t)-\bx_k(t)|^2}\nn\\
&=&8M_0+2\sum_{j,k=1,j\ne k}^N m_j m_k
\sum_{l=1,l\ne j,k}^N  \frac{(\mathbf
x_j(t)-\mathbf x_l(t))\cdot(\bx_j(t)-\bx_k(t))} {|\bx_j(t)-\bx_k(t)|^2}\nn\\
&&-2\sum_{j,k=1,j\ne k}^N m_j m_k
\sum_{l=1,l\ne j,k}^N  \frac{(\mathbf
x_k(t)-\mathbf x_l(t))\cdot(\bx_j(t)-\bx_k(t))} {|\bx_j(t)-\bx_k(t)|^2}.
\eea
\bea\label{G11}
I&=&8M_0+2\sum_{j,k=1,j\ne k}^N m_j m_k
\sum_{l=1,l\ne j,k}^N  \frac{(\bx_j(t)-\bx_k(t))\cdot \left[(\mathbf
x_j(t)-\mathbf x_l(t))-(\mathbf
x_k(t)-\mathbf x_l(t))\right]} {|\bx_j(t)-\bx_k(t)|^2}\nn\\
&=&8M_0+2\sum_{j,k=1,j\ne k}^N (N-2)m_j m_k = 8M_0+4(N-2)M_0=4NM_0, \qquad t\ge0.
\eea
Plugging \eqref{G12} into \eqref{H135},  we get
\be\label{H169}
\frac{dH_1(\mathbf{X}(t),t)}{dt}=-4NM_0+4NM_0=0, \quad t\ge0,
\ee
which immediately implies the left equation in \eqref{fi11} by noting
the initial condition (\ref{GLEi}).

Similarly, differentiating \eqref{H13} (with $\bX=\bX(t)$) with respect to $t$,  we get
\bea \label{H1356}
\frac{dH_3(\mathbf{X}(t),t)}{dt}&=&\left.\nabla_{\mathbf{X}}H_3(\bX,t)
\cdot\dot{\mathbf{X}}+\frac{\partial H_3(\bX,t)}{\partial t}\right|_{\bX=\bX(t)}\nn\\
&=&-4(N-2)M_0+2\sum_{1\leq j<l\leq N}(\mathbf
x_j(t)+\mathbf x_l(t))\cdot(\dot{\mathbf x}_j(t)+\dot {\mathbf x}_l(t)),
\quad t\ge0.
\eea
Similar to \eqref{G12}, noting \eqref{GLE} and \eqref{mjml}, we get
\bea\label{G13}
II&:=&2\sum_{1\leq j<l\leq N}(\mathbf
x_j(t)+\mathbf x_l(t))\cdot(\dot{\mathbf x}_j(t)+\dot {\mathbf x}_l(t))=\sum_{j,l=1,j\ne l}^N(\mathbf
x_j(t)+\mathbf x_l(t))\cdot(\dot{\mathbf x}_j(t)+\dot {\mathbf x}_l(t))\nn\\
&=&2\sum_{j,l=1,j\ne l}^N(\mathbf
x_j(t)+\mathbf x_l(t))\cdot\left[\sum_{k=1,k\ne j}^N m_j m_k \frac{\bx_j(t)-\bx_k(t)} {|\bx_j(t)-\bx_k(t)|^2}+\sum_{k=1,k\ne l}^N m_l m_k \frac{\bx_l(t)-\bx_k(t)} {|\bx_l(t)-\bx_k(t)|^2}\right]\nn\\
&=&2\sum_{j,k=1,j\ne k}^N m_j m_k
\sum_{l=1,l\ne j,k}^N  \frac{(\mathbf
x_j(t)+\mathbf x_l(t))\cdot(\bx_j(t)-\bx_k(t))} {|\bx_j(t)-\bx_k(t)|^2}\nn\\
&&+2\sum_{l,k=1,l\ne k}^N m_l m_k
\sum_{j=1,l\ne l,k}^N  \frac{(\mathbf
x_j(t)+\mathbf x_l(t))\cdot(\bx_l(t)-\bx_k(t))} {|\bx_l(t)-\bx_k(t)|^2}\nn\\
&=&2\sum_{j,k=1,j\ne k}^N m_j m_k
\sum_{l=1,l\ne j,k}^N  \frac{(\mathbf
x_j(t)+\mathbf x_l(t))\cdot(\bx_j(t)-\bx_k(t))} {|\bx_j(t)-\bx_k(t)|^2}\nn\\
&&-2\sum_{j,k=1,j\ne k}^N m_j m_k
\sum_{l=1,l\ne j,k}^N  \frac{(\mathbf
x_k(t)+\mathbf x_l(t))\cdot(\bx_j(t)-\bx_k(t))} {|\bx_j(t)-\bx_k(t)|^2}\nn\\
&=&2\sum_{j,k=1,j\ne k}^N m_j m_k
\sum_{l=1,l\ne j,k}^N  \frac{(\bx_j(t)-\bx_k(t))\cdot \left[(\mathbf
x_j(t)+\mathbf x_l(t))-(\mathbf
x_k(t)+\mathbf x_l(t))\right]} {|\bx_j(t)-\bx_k(t)|^2}\nn\\
&=&2\sum_{j,k=1,j\ne k}^N (N-2)m_j m_k =4(N-2)M_0, \qquad t\ge0.
\eea
Plugging \eqref{G13} into \eqref{H1356}, we get
\be\label{H1696}
\frac{dH_3(\mathbf{X}(t),t)}{dt}=-4(N-2)M_0+4(N-2)M_0=0, \qquad t\ge0,
\ee
which immediately implies \eqref{fi12} by noting
the initial condition (\ref{GLEi}).

From \eqref{H12} and \eqref{H13}, it is easy to see that
\be \label{H267}
H_2(\mathbf{X}(t),t)=\frac{1}{2(N-1)}\left[
H_1(\mathbf{X}(t),t)+H_3(\mathbf{X}(t),t)\right],
\qquad  t\ge0.
\ee
Differentiating \eqref{H267} with respect to $t$, noticing \eqref{H169} and
\eqref{H1696}, we have
\be
\frac{dH_2(\mathbf{X}(t),t)}{dt}=\frac{1}{2(N-1)}\left[
\frac{dH_1(\mathbf{X}(t),t)}{dt}+\frac{dH_3(\mathbf{X}(t),t)}{dt}\right]=0,
\qquad t\ge0,\nonumber
\ee
which immediately implies the right equation in \eqref{fi11} by noting
the initial condition (\ref{GLEi}). Therefore $H_1(\bX, t)$, $H_2(\bX, t)$ and $H_3(\bX, t)$  are three non-autonomous first integrals of
the ODEs (\ref{GLE}).
\end{proof}

\subsection{Global existence in the case with the same winding number}
Let $m_0=+1$ or $-1$ be fixed.
When the $N$ quantized vortices have the same winder number, e.g. $m_0$,  we have
\begin{theorem}\label{NNC}
Suppose the $N$ vortices have the same winding number, i.e. $m_j=m_0$ for $1 \le j\le N$ in \eqref{GLE}, then
$T_{\rm max}=+\infty$, i.e. there is no finite time collision among the $N$ quantized vortices. In addition, at least two
vortices move to infinity as $t\rightarrow +\infty$.
\end{theorem}
\begin{proof} The proof will be proceeded by the method of contradiction.
Assume $0<T_{\rm max}<+\infty$, i.e. there exist $M$ ($2\le M\le N$) vortices
(without loss of generality, we assume here that they are $\mathbf
x_1,\ldots,\mathbf x_M$)
that collide at a fixed point $\bx^0\in{\mathbb R}^2$ and the rest
$N-M$ vortices are all away from this point. Taking $t=T_{\rm max}$ in the left equation in \eqref{fi11}, noting \eqref{mjml}, \eqref{H12} and $|N^+-N^-|=N$, we get
\beas
0&<&H_1^0=H_1(\bX(T_{\rm max}),T_{\rm max})=-4NM_0T_{\rm max}
+\sum_{1\leq j<l\leq
N}|\mathbf{x}_j(T_{\rm max})-\mathbf{x}_l(T_{\rm max})|^2 \nn\\
&=&-2N^2(N-1)T_{\rm max}
+\sum_{j=M+1}^N|\mathbf{x}_j(T_{\rm max})-\bx^0|^2+ \sum_{M+1\leq j<l\leq
N}|\mathbf{x}_j(T_{\rm max})-\mathbf{x}_l(T_{\rm max})|^2,
\eeas
which immediately implies that $2\le M<N$.
Denote the non-empty sets $I=\{1,\ldots,M\}$ and $J=\{M+1,\ldots,N\}$, and
define
\be\label{DIt10}
D_I(t)=\sum_{1\leq j<l\leq
M}d_{jl}^2(t), \qquad  d_{I,J}(t)=\min_{j\in I, l\in
J}d_{jl}(t),\qquad 0\le t\le T_{\rm max},
\ee
where
\begin{equation}
d_{jl}(t)=|\bx_j(t)-\bx_l(t)|, 
\quad t\ge0, \quad 1\le j< l\le N.\nonumber
\end{equation}
Then we have
\be
\lim_{t\rightarrow T_{\rm max}^-}D_{I}(t)=0,\qquad  \mathop{\underline
{\lim } }\limits_{t \rightarrow T_{\rm max}^- }d_{I,J}(t)>0,\nonumber
\ee
which yields
\be
d_1:=\min_{0\le t\le T_{\rm max}} d_{I,J}(t)>0.\nonumber
\ee
Choose $\varepsilon= \frac{Md_1}{3(N-M)}>0$, then there exists a
$0<T_1<T_{\rm max}$ such that
\[  0\le D_{I}(t)<\varepsilon, \qquad T_1\le t\le T_{\rm max}.\]
Differentiating \eqref{DIt10} with respect to $t$, we obtain
\bea
\dot D_I(t)&=&2\sum_{1\leq j<l\leq M}(\mathbf
x_j(t)-\mathbf x_l(t))\cdot(\dot{\mathbf x}_j(t)-\dot {\mathbf x}_l(t))=\sum_{j,l=1,j\ne l}^M(\mathbf
x_j(t)-\mathbf x_l(t))\cdot(\dot{\mathbf x}_j(t)-\dot {\mathbf x}_l(t))\nn\\
&=&2\sum_{j,l=1,j\ne l}^M(\mathbf
x_j(t)-\mathbf x_l(t))\cdot\left[\sum_{k=1,k\ne j}^N  \frac{\bx_j(t)-\bx_k(t)} {d_{jk}^2(t)}-\sum_{k=1,k\ne l}^N  \frac{\bx_l(t)-\bx_k(t)} {d_{lk}^2(t)}\right]\nn\\
&=&2\sum_{j,l=1,j\ne l}^M(\mathbf
x_j(t)-\mathbf x_l(t))\cdot\left[\sum_{k=1,k\ne j}^M \frac{\bx_j(t)-\bx_k(t)} {d_{jk}^2(t)}- \sum_{k=1,k\ne l}^M \frac{\bx_l(t)-\bx_k(t)} {d_{lk}^2(t)}\right]\nn\\
&&+2\sum_{j,l=1,j\ne l}^M(\mathbf
x_j(t)-\mathbf x_l(t))\cdot\sum_{k=M+1}^N  \left[\frac{\bx_j(t)-\bx_k(t)} {d_{jk}^2(t)}-\sum_{k=M+1}^N \frac{\bx_l(t)-\bx_k(t)} {d_{lk}^2(t)}\right].\nonumber
\eea
Similar to \eqref{H135} via \eqref{G12} (with details omitted here for brevity), we get
\bea\label{da}
\dot D_I(t)&=&4\sum_{1\leq j<l \leq M}\left[M+(\mathbf
x_j(t)-\mathbf x_l(t))\cdot \sum_{k=M+1}^N\left(\frac{ \bx_j(t)-\bx_k(t)}{d_{jk}^2(t)}-\frac{\bx_l(t)-\bx_k(t)}{d_{lk}^2(t)}\right)\right]\nn\\
&\ge&4\sum_{1\leq j<l \leq M}\left[M-d_{jl}(t)\sum_{k=M+1}^N\left(\frac{1}{d_{jk}(t)}+
\frac{1}{d_{lk}(t)}\right)\right]\ge4\sum_{1\leq j<l \leq
M}\left[M-2\sum_{k=M+1}^N\frac{\varepsilon}{d_1}\right]\nn\\
&=&2M(M-1)\left(M-2(N-M)\frac{\varepsilon}{d_1}\right) >0,
\qquad T_1\le t\le T_{\rm max},
\eea
which immediately implies that $0=D_I(T_{\rm max}^-)\ge D_I(T_1)>0.$
This is a contradiction, and thus $T_{\rm max}=+\infty$, i.e.
there is no finite time collision among the $N$ quantized vortices.

Noticing $M_0=\frac{1}{2}N(N-1)>0$, combining \eqref{H12} and \eqref{fi11},
we get $\sum_{j=1}^N|\bx_j(t)|^2=
\lim_{t\to+\infty}H_2^0+4M_0t\to +\infty$ when
$t\to+\infty$.  Hence there exists an $1\le i_0\le N$ such that
$|\bx_{i_0}(t)|\to +\infty$ when $t\to+\infty$.
Due to the conservation of mass center, i.e. $\bar \bx(t):=\frac{1}{N}\sum_{j=1}^N \bx_j(t)\equiv \bar \bx(0)$,
there exists at least another $1\le j_0\neq i_0\le N$
such that $|\bx_{j_0}(t)|\to +\infty$ when $t\to+\infty$.
Thus there exist at least two
vortices  move to infinity when $t\rightarrow +\infty$.
\end{proof}

  Define
\be\label{dmin11}
d_{\rm min}(t)=\min_{1\leq j<l\leq N}d_{jl}(t),\quad D_{\rm min}(t)=d_{\rm min}^2(t), \quad D_{jl}(t)=d_{jl}^2(t), \quad 1\le j\ne l\le N,\quad t\ge0.
\ee
Then it is easy to see that $d_{\rm min}(t)$ and $D_{\rm min}(t)$ are continuous and piecewise smooth functions. In addition, for $1\le j<l\le N$, noting \eqref{GLE}, we have
\bea\label{Djlt1}
\dot D_{jl}(t)
&=&2(\mathbf
x_j(t)-\mathbf x_l(t))\cdot(\dot{\mathbf x}_j(t)-\dot {\mathbf x}_l(t))\nn\\
&=&4(\mathbf
x_j(t)-\mathbf x_l(t))\cdot\left[\sum_{k=1,k\ne j}^N \frac{\bx_j(t)-\bx_k(t)} {d_{jk}^2(t)}-\sum_{k=1,k\ne l}^N  \frac{\bx_l(t)-\bx_k(t)} {d_{lk}^2(t)}\right]\nn\\
&=&4\left[2+(\mathbf
x_j(t)-\mathbf x_l(t))\cdot\sum_{k=1,k\ne j,l}^N\left(\frac{\bx_j(t)-\bx_k(t)} {d_{jk}^2(t)}-
\frac{\bx_l(t)-\bx_k(t)} {d_{lk}^2(t)}\right)\right], \quad t\ge0.\quad
\eea


When the $N$ vortices are initially collinear,
we have

\begin{theorem}\label{cdi}
Suppose the $N$ vortices have the same winding number, i.e. $m_j=m_0$ for $1 \le j\le N$ in \eqref{GLE}, and the initial data $\mathbf{X}^0$ in (\ref{GLEi}) is collinear, then
$d_{\rm min}(t)$ and $D_{\rm min}(t)$ are monotonically increasing functions.
\end{theorem}
\begin{proof}
Since $\mathbf{X}^0$ is collinear,
there exist $\bx^0\in {\mathbb R}^2$ and a unit vector ${\bf e}\in {\mathbb R}^2$ such that  $\bX^0\in S_{{\bf e}}(\bx^0)$, by Lemma \ref{lmcolin},
we know that $\bX(t)\in S_{{\bf e}}(\bx^0)$ for $t\ge0$.
Thus there exist $a_j(t)$ ($1\le j\le N$) satisfying $a_j(t)\ne a_l(t)$ for
$1\le j<l\le N$ such that
\be\label{colin2}
\bx_j(t)=\bx^0+a_j(t) {\bf e}, \qquad t\ge0, \qquad 1\le j\le N.
\ee
Taking $0\le t_0<t_1$ such that $d_{\rm min}(t)$ is smooth on $[t_0,t_1)$,
without loss of generality, we assume that there exists $1\leq i_0\leq N-1$ (otherwise by re-ordering) such that
\be\label{ajlt}
a_1(t)<a_2(t)<\ldots <a_{i_0}(t)<a_{i_0+1}(t)<\ldots<a_N(t), \quad d_{\rm min}(t)=d_{i_0,i_0+1}(t), \quad t_0\le t<t_1.
\ee
Plugging $j=i_0$ and $l=i_0+1$ into \eqref{Djlt1} and noting \eqref{ajlt},
 \eqref{colin2} and \eqref{dmin11}, we gave
\bea
\dot D_{\rm min}(t)&=&\dot D_{i_0,i_0+1}(t)=4\left[2+d_{\rm min}(t)\sum_{k=1,k\ne i_0,i_0+1}^N\left(\frac{1} {d_{i_0k}(t)}-
\frac{1} {d_{i_0+1,k}(t)}\right)\right]\nn\\
&=&4\left[2-\sum_{k=1}^{i_0-1}
\frac{d_{\rm min}^2(t)} {d_{i_0k}(t)d_{i_0+1,k}(t)}-\sum_{k=i_0+2}^N
\frac{d_{\rm min}^2(t)} {d_{i_0k}(t)d_{i_0+1,k}(t)}\right]\nn\\
&\ge&4\left[2-\sum_{k=1}^{i_0-1}\frac{1}
{(i_0-k)(i_0+1- k)}-\sum_{k=i_0+2}^{N}\frac{1}
{(k-i_0)(k-i_0-1)}\right]\nn\\
&=&4\left(\frac{1}{i_0}+\frac{1}{N-i_0}\right)>0, \qquad t_0\le t<t_1.\nonumber
\eea
Here we used $\frac{d_{jl}(t)}{d_{\rm min}(t)}\ge |j-l|$ for $1\le j<l\le N$ by noting \eqref{ajlt}. Thus $D_{\rm min}(t)$ (and $d_{\rm min}(t)$) is
a monotonically increasing function over $[t_0,t_1)$. Therefore,
$D_{\rm min}(t)$ (and $d_{\rm min}(t)$) is
a monotonically increasing function over its every piecewise smooth interval.
Due to that it is a continuous function, thus
$D_{\rm min}(t)$ (and $d_{\rm min}(t)$) is
a monotonically increasing function for $t\ge0$.
\end{proof}

Similarly,  when $2\le N\le 4$ and $\bX^0\in{\mathbb R}_*^{2\times N}$, we have

\begin{theorem}\label{N234}
Suppose $2\le N\le 4$ and the $N$ vortices have the same winding number, i.e. $m_j=m_0$ for $1 \le j\le N$ in \eqref{GLE}, then
$d_{\rm min}(t)$ and $D_{\rm min}(t)$ are monotonically increasing functions.
\end{theorem}
\begin{proof}
Taking $0\le t_0<t_1$ such that $d_{\rm min}(t)$ is smooth on $[t_0,t_1)$,
without loss of generality, we assume that $d_{\rm min}(t)=d_{12}(t)$ for
$t_0\le t<t_1$ (otherwise by re-ordering).
Taking $j=1$ and $l=2$ in \eqref{Djlt1}, we get for $t_0\le t<t_1$
\be
\dot D_{\rm min}(t)=\dot D_{12}(t)
=4\left[2+(\mathbf
x_1(t)-\mathbf x_2(t))\cdot\sum_{k=1,k\ne 1,2}^N\left(\frac{\bx_1(t)-\bx_k(t)} {d_{1k}^2(t)}-
\frac{\bx_2(t)-\bx_k(t)} {d_{2k}^2(t)}\right)\right].\nonumber
\ee
When $N=2$ or $3$, noting $0<d_{12}(t)\le d_{jl}(t)$ for $1\le j\ne l \le N$,
we get
\be
\dot D_{\rm min}(t)>
4\left[2-
\sum_{k=1,k\ne 1,2}^N\left(\frac{d_{12}(t)} {d_{1k}(t)}+
\frac{d_{12}(t)} {d_{2k}(t)}\right)\right]
\ge4\left(2-2(N-2)\right)\ge0, \qquad t_0\le t<t_1,\nonumber
\ee
which implies that $D_{\rm min}(t)$ and $d_{\rm min}(t)$
are monotonically increasing functions over $t\in [t_0,t_1]$.
When $N=4$, without loss of generality, we can assume
\begin{equation*}
d_{12}(t)\leq d_{13}(t)\leq d_{23}(t), \quad  d_{12}(t)\leq d_{14}(t)\leq d_{24}(t),\qquad t_0\le t<t_1,
\end{equation*}
then we get
\be
\dot D_{\rm min}(t)>
4\left[2-
\left(\frac{d_{12}(t)} {d_{23}(t)}+
\frac{d_{12}(t)} {d_{24}(t)}\right)\right] \ge0, \qquad t_0\le t<t_1,\nonumber
\ee
which implies that $D_{\rm min}(t)$ and $d_{\rm min}(t)$
are monotonically increasing functions over $t\in [t_0,t_1]$.
\end{proof}

\begin{remark}
When $N\ge5$ and the initial data $\bX^0\in
{\mathbb R}_*^{2\times N}$ is not collinear, $d_{\rm min}(t)$ might not be a
monotonically increasing function, especially when $0\le t\ll 1$. Based on our extensive numerical results, for any given  $\bX^0\in
{\mathbb R}_*^{2\times N}$, there exits a constant $T_{0}>0$ depending on $\bX^0$ such that $d_{\rm min}(t)$
is a monotonically increasing function when $t\ge T_{0}$.
Rigorous mathematical justification is ongoing.
\end{remark}


\subsection{Finite time collision in the case with opposite winding numbers}
When the $N$ vortices have opposite winding numbers, we have

\begin{theorem}\label{fi}
Suppose the $N$ vortices have opposite winding numbers,
i.e. $|N^+-N^-|<N$, we have

(i)  If $M_0<0$, finite time collision happens, i.e.
$0<T_{\rm max}<+\infty$,
and there exists a collision cluster among the $N$ vortices. In addition,
$T_{\rm max}\leq T_a:=-\frac{H_1^0}{4NM_0}$.

(ii) If $M_0=0$, then the solution of {\rm(\ref{GLE})} is
bounded, i.e.
\be\label{bxjb1}
|\bx_j(t)|\le \sqrt{H_2^0}=\sqrt{\sum_{j=1}^N|\bx_j^0|^2}, \qquad t\ge0, \qquad 1\le j\le N.
\ee

(iii) If $M_0>0$ and there is no finite time collision, i.e.
$T_{\rm max}=+\infty$, then at least two
vortices move to infinity as $t\rightarrow +\infty$.

(iv) 
Let $I\subseteq \{1,2,\ldots, N\}$ be a set with $M$ ($2\le M\le N$) elements.
If the collective winding number of $I$ defined as $M_1:=\frac{1}{2}\sum_{j,l\in I,j\ne l} m_jm_l\ge0$, then the set
of vortices $\{\bx_j(t) \ |\ j\in I\}$ cannot be a collision cluster
among the $N$ vortices  for $0\le t\le T_{\rm max}$.
\end{theorem}

\begin{proof} (i) Combining \eqref{H12} and \eqref{fi11}, we get
\begin{equation} \label{bjl297}
0\le \sum_{1\leq j<l\leq
N}|\bx_j(t)-\bx_l(t)|^2=4NM_0t+H_1^0,\qquad t\ge0.
\end{equation}
If $M_0<0$, when $t\to T_a=-\frac{H_1^0}{4NM_0}$, we have
$4NM_0t+H_1^0\to 0$. Thus finite time collision happens at $t=T_{\rm max}\le T_a<+\infty$.

(ii) If $M_0=0$, combining \eqref{H12} and \eqref{fi11}, we get
\be
0\le |\bx_j(t)|^2\le \sum_{j=1}^N|\bx_j(t)|^2\equiv H_2^0=\sum_{j=1}^N|\bx_j^0|^2, \qquad t\ge0, \nonumber
\ee
which immediately implies \eqref{bxjb1}.

(iii) If $ M_0>0$ and $T_{\rm max}=+\infty$, then there exists no finite time collision cluster among the $N$ vortices. The proof can be proceeded
similarly as the last part in Theorem \ref{NNC} and details are omitted here for brevity.


(iv) When $M=N$, for any given $\bX^0\in {\mathbb R}_*^{2\times N}$,  we get $H_1^0>0$. Noting \eqref{bjl297} and $M_0=M_1\ge0$, we have
\[\sum_{1\leq j<l\leq
N}|\bx_j(t)-\bx_l(t)|^2=4NM_0t+H_1^0\ge H_1^0>0, \qquad 0\le t\le
T_{\rm max},
\]
which immediately implies that the $N$ vortices cannot be a
collision cluster when $t\in [0,T_{\rm max}]$ for any given $\bX^0$.

When $2\le M<N$ and $N\ge3$, without loss of generality,
we assume $I=\{1,\ldots,M\}$ and denote $J=\{M+1,\ldots,N\}$.
Thus $M_1=\sum_{1\le j< l\le M}m_jm_l\ge0$.
We will proceed the proof by the method of
contradiction. Assume that the $M$ vortices $\bx_1,\ldots,\bx_M$ collide
at $\bx^0\in {\mathbb R}^2$ when $t\to T_c$ satisfying $0<T_c\le T_{\rm max}$, i.e.  $\bx_j(t)\to \bx^0$ when $t\to T_{c}^-$ for $1\le j\le M$
and $|\bx_j(t)-\bx^0|>0$ when $t\to T_{c}^-$ for $M+1\le j\le N$.
Denote $d_2:=\displaystyle\min_{j\in
J}\min_{0\le t \le T_c}|\mathbf x_j(t)-\mathbf x_0|^2$ and we have
$d_2>0$.
Since $\displaystyle\lim_{t\rightarrow T_{c}}D_{I}(t)=0$,
 there exists $0<T_1<T_{c}$, such that $D_{I}(t)<\frac{d_2}{2}$ and
 $d_{I,J}(t)>\frac{d_2}{2}$ for
 $t\in[T_1,T_{c})$.  Choose
$T_2\in[T_1,T_{c})$, such that
$$
0<T_{c}-T_2<\frac{d_2}{8M(M-1)(N-M)}.
$$
Since $D_{I}(t)$ is a continuous function, there exists $T_3\in[T_2,
T_{c}]$, such that $
D_{I}(T_3)=\max_{t\in[T_2,T_{c}]}D_{I}(t)>0$. Similar to \eqref{da}, we have
\bea
D_{I}(T_3)&=&D_{I}(T_3)-D_{I}(T_c)=
-\int_{T_3}^{T_{c}}\frac{d}{dt}D_I(t)dt \nn\\
&=&-4\int_{T_3}^{T_{c}}\sum_{1\leq j<l \leq M}(\mathbf x_j(t)-\mathbf
x_l(t))\cdot \sum_{k=M+1}^N m_k\left[m_j\frac{\mathbf
x_j(t)-\mathbf x_k(t)}{d_{jk}^2(t)}-m_l\frac{\mathbf x_l(t)-\mathbf
x_k(t)}{d_{lk}^2(t)}\right]dt\nn\\
&&-4MM_1(T_c-T_3)\nn\\
&\le&4\int_{T_3}^{T_{c}}\sum_{1\leq j<l \leq M}\sum_{k=M+1}^N
\left(\frac{d_{jl}(t)}{d_{jk}(t)}+
\frac{d_{jl}(t)}{d_{lk}(t)}\right)dt
\le16\int_{T_3}^{T_{c}}\sum_{1\leq j<l \leq
M}\sum_{k=M+1}^N\frac{D_{I}(T_3)}{d_2}dt\nn\\
&=&\frac{8M(M-1)(N-M)(T_c-T_{3})}{d_2}D_{I}(T_3)\nn\\
&\le&\frac{8M(M-1)(N-M)(T_c-T_{2})}{d_2}D_{I}(T_3) <D_{I}(T_3).\nonumber
\eea
This is a contradiction and thus the set
of vortices $\{\bx_j(t) \ |\ j\in I\}$ cannot be a collision cluster
among the $N$ vortices  for $0\le t\le T_{\rm max}$.
\end{proof}

\begin{proposition}\label{total}
If the $N$ vortices be a collision cluster at
$0<T_{\rm max}<+\infty$ under a given initial data
$\bX^0\in {\mathbb R}_*^{2\times N}$,
then we have
\be\label{M0clu56}
M_0<0, \qquad H_1^0=N H_2^0, \qquad H_3^0=(N-2)H_2^0.
\ee
\end{proposition}

\begin{proof}
Due to the conservation of mass center and $\bX^0\in
{\mathbb R}_*^{2\times N}$,
we get
\be\label{xj9765}
\bar{\bx}(t)\equiv \bar\bx^0 \Longrightarrow \lim_{t\to T_{\rm max}^-} \mathbf x_j(t)=\bx_j(\Tmax)=\bar\bx^0, \qquad 1\le j\le N.
\ee
Plugging \eqref{xj9765} into \eqref{H12} and \eqref{fi11}, we get
\be\label{H1m76}
H_1(\bX(T_{\rm max}),T_{\rm max})=\sum_{1\le j<l\le N}|\bx_j(\Tmax)-\bx_l(\Tmax)|^2-4NM_0T_{\rm max}=-4NM_0T_{\rm max}=H_1^0.
\ee
Similarly, we have
\be\label{H2376}
H_2(\bX(T_{\rm max}),T_{\rm max})=-4M_0T_{\rm max}=H_2^0,\quad
H_3(\bX(T_{\rm max}),T_{\rm max})=-4(N-2)M_0T_{\rm max}=H_3^0.
\ee
Combining \eqref{H1m76} and \eqref{H2376}, we obtain \eqref{M0clu56}.
\end{proof}

\begin{proposition}\label{equil}
If the ODEs \eqref{GLE} admits an equilibrium solution, then
$N\ge4 $ is a square of an integer, i.e. $N=(N^+-N^-)^2$ and
\be\label{Np12}
1\le N^+=\frac{1}{2}\left(N\pm \sqrt{N}\right)<N,\qquad
 1\le N^-=N-N^+<N.
 \ee
\end{proposition}

\begin{proof} Assume $\bX(t)\equiv \bX^0\in
{\mathbb R}_*^{2\times N}$ be
an equilibrium solution of (\ref{GLE}), noting \eqref{H12} and \eqref{fi11}, we get
\begin{equation}\label{32}
H_1(\bX(t),t)=H_1^0-4NM_0t \equiv H_1^0, \qquad t\ge0.
\end{equation}
Thus $M_0=0$. Noting \eqref{mjml}, we have
\be\label{N123}
4\le N=(N^{+}-N^{-})^2=(2N^+-N)^2=(2N^--N)^2.
\ee
Thus $N\ge4 $ is a square of an integer and we obtain \eqref{Np12}
by solving \eqref{N123}.
\end{proof}

\begin{remark}
When $N=4$, an equilibrium solution of \eqref{GLE} was constructed
in \cite{Zhang1,Zhang2} by taking $m_4=-1$, $\bx_4^0=(0,0)^T$ and
$m_1=m_2=m_3=1$, $\bx_j^0$ located in the vertices of a right triangle
centered at the origin. Here we want to remark
that any equilibrium solution of \eqref{GLE} is dynamically unstable.
\end{remark}

\section{Interaction patterns of a cluster with $3$ quantized vortices}
\setcounter{equation}{0}
\setcounter{figure}{0}

 In this section, we assume $N=3$ in (\ref{GLE}) and (\ref{GLEi}).

\subsection{Structural/obital stability in the case with the same winding number}
Assume that $m_1=m_2=m_3$ and by Theorem
\ref{NNC}, we know the ODEs (\ref{GLE}) with (\ref{GLEi})
is globally well-posed, i.e. $T_{\rm max}=+\infty$.

\begin{lemma}\label{invariants}
If the initial data $\bX^0\in {\mathbb R}_*^{2\times 3}$ in (\ref{GLEi}) with
$N=3$ is collinear, then one vortex moves to the mass center
$\bar\bx^0$  and the other two vortices repel with each other and move outwards to far field along the line when $t\to+\infty$.
\end{lemma}

\begin{proof} Since $\bX^0\in {\mathbb R}_*^{2\times 3}$ is collinear, there exist $\bx^0\in {\mathbb R}^2$ and a unit vector
${\bf e}\in {\mathbb R}^2$ such that
\be
\bx_j^0=\bx^0+a_j^0 {\bf e}, \qquad j=1,2,3.\nonumber
\ee
Without loss of generality, we assume that
\be
 a_1^0<a_2^0<a_3^0, \qquad a_1^0<0, \qquad a_3^0>0, \qquad a_1^0+a_2^0+a_3^0=0.\nonumber
\ee
Based on the results in Lemma \ref{lmcolin} and Theorem \ref{NNC},
we know that there exist $a_j(t)$ ($j=1,2,3$) such that
\be\label{ajt346}
\bx_j(t)=\bx^0+a_j(t) {\bf e}, \qquad j=1,2,3,
\ee
satisfying
\be\label{ajt347}
 a_1(t)<a_2(t)<a_3(t),  \qquad a_1(t)+a_2(t)+a_3(t)\equiv0, \qquad t\ge0.
\ee
Plugging \eqref{ajt346} into \eqref{GLE} with $N=3$ and $m_1=m_2=m_3$, noting \eqref{ajt347}, we get
\bea
\dot a_1(t)&=&-\frac{2}{a_2(t)-a_1(t)}-\frac{2}{a_3(t)-a_1(t)}=
\frac{6a_1(t)}{[a_2(t)-a_1(t)][a_3(t)-a_1(t)]}<0,\nn \\
\dot a_2(t)&=&\frac{2}{a_2(t)-a_1(t)}-\frac{2}{a_3(t)-a_2(t)}
=\frac{-6a_2(t)}{[a_2(t)-a_1(t)][a_3(t)-a_2(t)]},\qquad \qquad t>0,\nn \\
\dot a_3(t)&=&\frac{2}{a_3(t)-a_1(t)}+\frac{2}{a_3(t)-a_2(t)}
=\frac{6a_3(t)}{[a_3(t)-a_1(t)][a_3(t)-a_2(t)]}>0,\nn
\eea
with the initial data
\be\label{init3478}
a_j(0)=a_j^0, \qquad j=1,2,3.
\ee
Thus $a_1(t)$ is a monotonically decreasing function and
$a_3(t)$ is a monotonically increasing function for $t\ge0$.
Let $\rho_2(t)=a_2^2(t)\ge0$, then we have
\be
\dot \rho_2(t)=\frac{-12\rho_2(t)}{[a_2(t)-a_1(t)][a_3(t)-a_2(t)]}<0,\qquad t>0,\nonumber
\ee
which immediately implies that $\rho_2(t)$ is a monotonically decreasing function and $\lim_{t\to +\infty} \rho_2(t)=0$. Thus we have
\be
\lim_{t\to +\infty} a_2(t)=0 \Longrightarrow \lim_{t\to +\infty}
\bx_2(t)=\bar\bx^0=\frac{1}{3}\sum_{j=1}^3 \bx_j^0.\nonumber
\ee
Thus the vortex $\bx_2(t)$ moves towards $\bar\bx^0$ along the line $S_{\bf e}(\bar\bx^0)$.
Based on the results in Theorem \ref{NNC}, we know that at least two vortices
must move to infinity when $t\to+\infty$. Thus we have
\be
 a_1(t)\to -\infty, \qquad a_3(t)\to +\infty \qquad
\hbox{when} \qquad t\to +\infty. \nonumber
\ee
Thus the other two vortices $\bx_1(t)$ and $\bx_3(t)$ repel with each other and move outwards to far field along the line $S_{\bf e}(\bx^0)$ when $t\to+\infty$.
\end{proof}

\begin{theorem}\label{th2}
Assume the initial data $\bX^0\in {\mathbb R}_*^{2\times 3}$ in \eqref{GLEi} with $N=3$ is not collinear,
then there exists a unit vector
${\bf e}\in {\mathbb R}^2$ such that
\be\label{dSt231}
\lim_{t\rightarrow+\infty} d_S(t):=\inf_{\bX\in S_{\bf e}^3(\bar\bx^0)}
\|\bX(t)-\bX\|_2=0.
\ee
\end{theorem}

\begin{proof} Without loss of
generality, as shown in Fig. \ref{vortex3}a, we assume $\bar\bx^0={\bf 0}$ and $d_{12}^0:=d_{12}(0)\le d_{13}^0:=d_{13}(0)\le d_{23}^0:=d_{23}(0)$.
Thus $0<\theta_3^0:=\theta_3(0)\le \theta_2^0:=\theta_2(0)\le
\theta_1^0:=\theta_1(0)<\pi$ satisfying
$\theta_1^0+\theta_2^0+\theta_3^0=\pi$ (cf. Fig. \ref{vortex3}a).
From \eqref{GLE} with $N=3$, we get
\bea\label{d1231}
\dot{d}_{12}(t)&=&\frac{4}{d_{12}(t)}+\frac{2\cos(\theta_{1}(t))}{d_{13}(t)}
+\frac{2\cos(\theta_2(t))}{d_{23}(t)},\quad
\dot\theta_{3}(t)=B(t)\left[d_{13}^2(t)+d_{23}^2(t)-2d_{12}^2(t)\right],\\
\label{d1232}
\dot{d}_{13}(t)&=&\frac{4}{d_{13}(t)}+\frac{2\cos(\theta_{1}(t))}{d_{12}(t)}
+\frac{2\cos(\theta_3(t))}{d_{23}(t)},\quad
\dot\theta_{2}(t)=B(t)\left[d_{12}^2(t)+d_{23}^2(t)-2d_{13}^2(t)\right],\qquad\\
\label{d1233}
\dot{d}_{23}(t)&=&\frac{4}{d_{23}(t)}+\frac{2\cos(\theta_{2}(t))}{d_{12}(t)}
+\frac{2\cos(\theta_3(t))}{d_{13}(t)},\quad
\dot\theta_{1}(t)=B(t)\left[d_{12}^2(t)+d_{13}^2(t)-2d_{23}^2(t)\right],
\eea
where $B(t):=4A(t)/(d_{12}(t)d_{13}(t)d_{23}(t))^2$ with
$A(t)$ denoting the area of the triangle with vertices
$\bx_1(t)$, $\bx_2(t)$ and $\bx_3(t)$.
Denote
\[\rho_{12}(t)=d_{12}^2(t),\quad
\rho_{13}(t)=d_{13}^2(t),\quad \rho_{23}(t)=d_{23}^2(t),\qquad
t\ge0.\]
From \eqref{d1231}-\eqref{d1233} and noting the initial data, we get
$\frac{\pi}{3}\le \theta_1(t)<\pi$ and $0<\theta_3(t)\le \frac{\pi}{3}$ are monotonically decreasing and increasing
functions, respectively, and
\be
\rho_{12}(t)\le \rho_{23}(t),\quad 0<\theta_3(t)\le\frac{\pi}{3}\le \theta_1(t)<\pi,
 \quad t\ge0; \qquad  \lim_{t\to+\infty}\theta_3(t)=\lim_{t\to+\infty}\theta_1(t)=\frac{\pi}{3}.
\ee
Combining this with $\theta_1(t)+\theta_2(t)+\theta_3(t)\equiv \pi$ for $t\ge0$, we get $\lim_{t\to+\infty}\theta_2(t)=\frac{\pi}{3}$, which immediately implies \eqref{dSt231}.
\end{proof}

\begin{figure}[bhtp]
\begin{minipage}[b]{0.4\textwidth}
\centering
\includegraphics[height=5cm,width=6.5cm,angle=0]{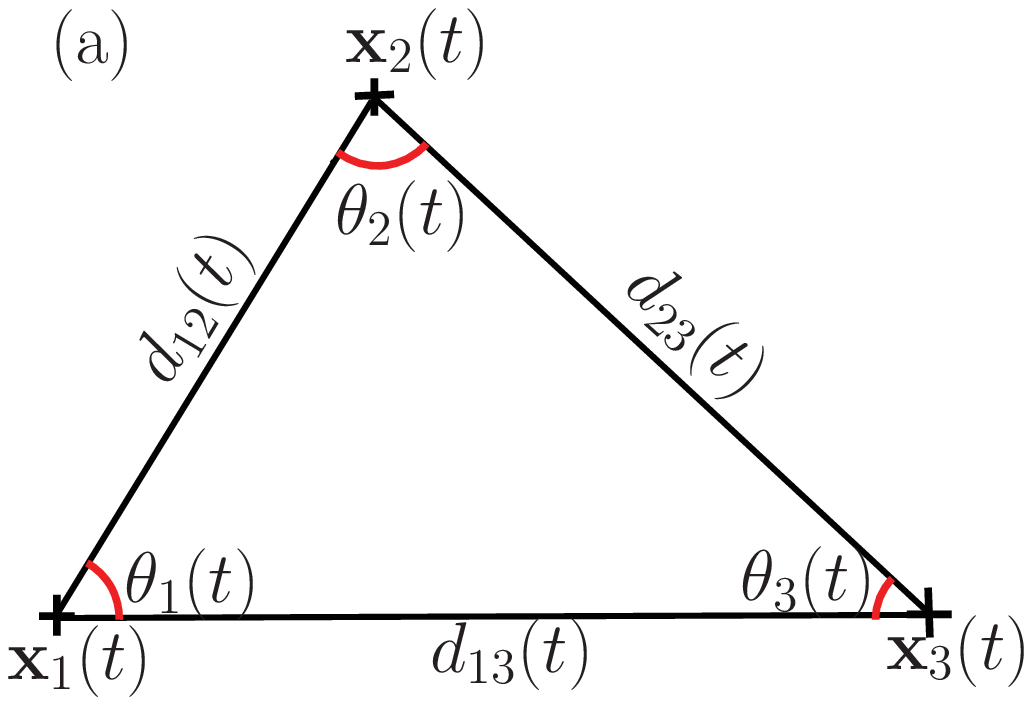}
\end{minipage} \quad \quad
\begin{minipage}[b]{0.4\textwidth}
\centering
\includegraphics[height=5.8cm,width=6.5cm,angle=0]{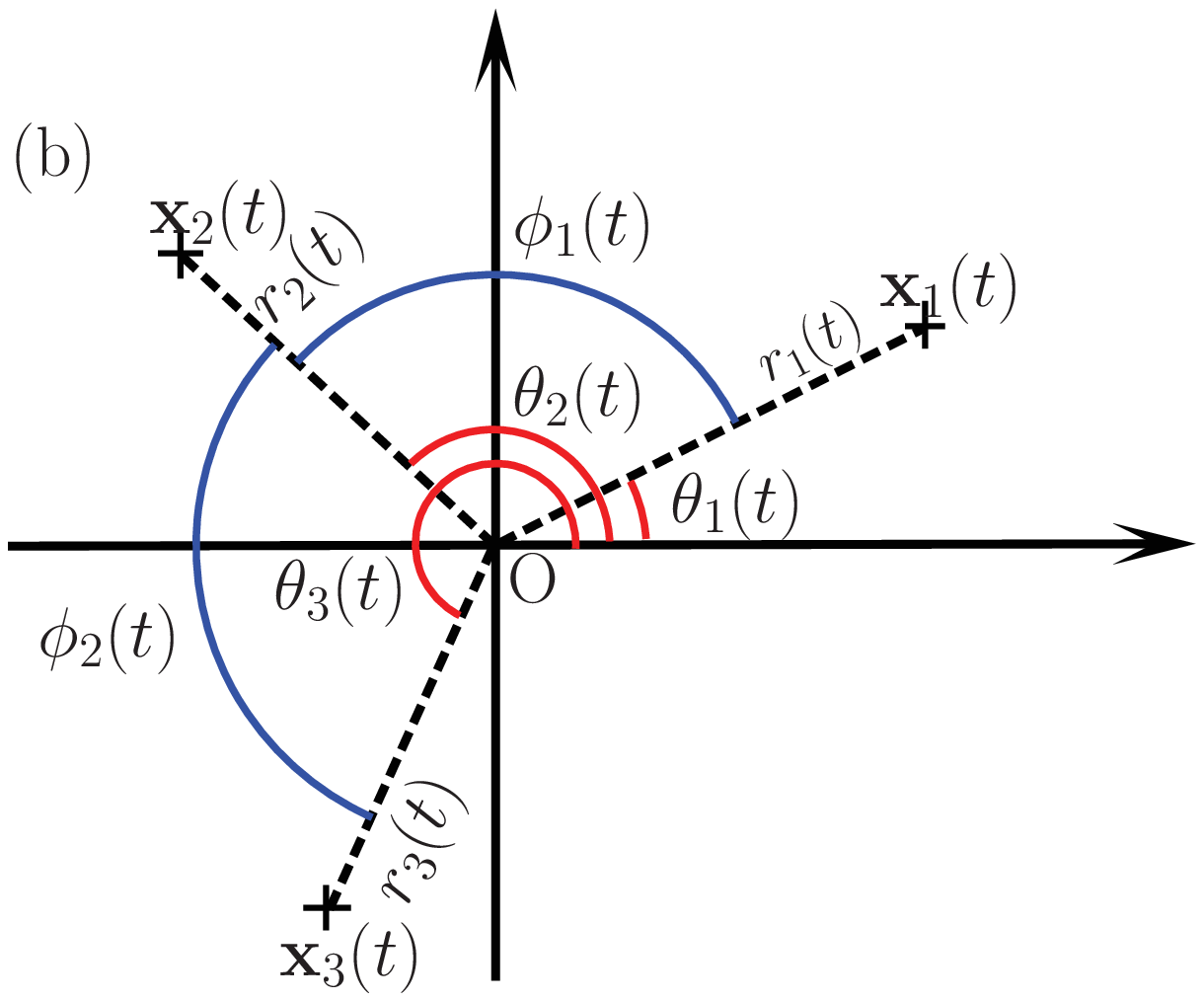}
\end{minipage}\\
\vfill
\begin{minipage}[l]{0.4\textwidth}
\centering
\includegraphics[height=5cm,width=6.5cm,angle=0]{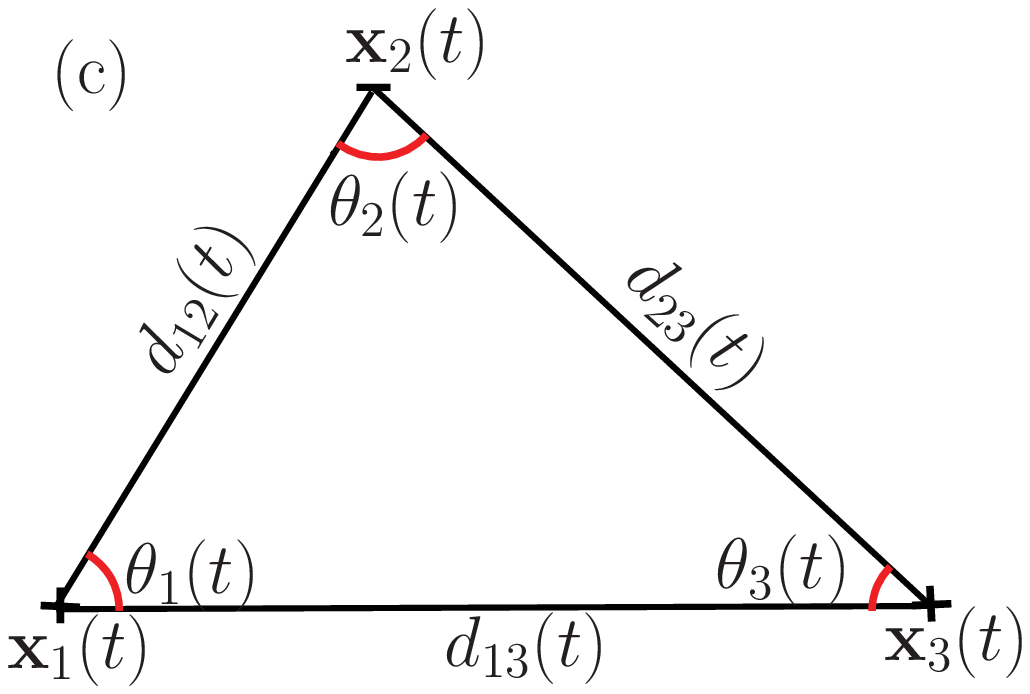}
\end{minipage}
\caption{Interaction of $3$ vortices
with the same winding number (a and b) and
opposite winding numbers (c).}\label{vortex3}
\end{figure}

For $\theta\in {\mathbb R}$ and
$\bX=(\bx_1,\ldots,\bx_N) \in S_{{\bf e}}^N({\bf 0},\theta_0)$, define
$Q(\theta)\bX:=(Q(\theta)\bx_1,\ldots,Q(\theta)\bx_N)$.

\begin{definition}
For the self-similar solution $\tilde \bX(t)=\sqrt{r_0^2+2(N-1)t}\,\tilde\bX^0$ with
$\tilde\bX^0=(\tilde\bx_1^0,\ldots,\tilde\bx_N^0)\in S_{{\bf e}}^N({\bf 0},\theta_0)$ and
$r_0=|\tilde \bx_1^0|$ of the ODEs \eqref{GLE} with $m_1=\ldots=m_N$,
if for any $\varepsilon>0$, there exists $\delta>0$ such that,
when the initial data $\bX^0$ in \eqref{GLEi} satisfies $\|\bX^0- \tilde\bX^0\|_2<\delta$, the solution $\bX(t)$ of the ODEs \eqref{GLE}
with \eqref{GLEi} satisfies
\be
\sup_{t\geq 0}\ \inf_{r>0,\ \theta\in[0,2\pi)}\ \left\|\bX(t)-\bar{\bx}^0 -r Q(\theta)\tilde \bX(t)\right\|<\varepsilon,\nonumber
\end{equation}
then  the self-similar solution $\tilde \bX(t)$ is called as
orbitally stable.
\end{definition}

\begin{theorem}
For any $\theta_0\in {\mathbb R}$ and $\tilde\bX^0=(\tilde\bx_1^0,\tilde\bx_2^0,\tilde\bx_3^0)\in S_{{\bf e}}^3({\bf 0},\theta_0)$, the solution $\tilde{\mathbf X}(t)=\sqrt{4t+r_0^2}\,
\tilde{\mathbf X}^0$ with $r_0=|\tilde \bx_1^0|$ of the ODEs \eqref{GLE}
with $N=3$ and $m_1=m_2=m_3$ is orbitally stable.
\end{theorem}

\begin{proof} By using Lemma \ref{invar12}, without loss of generality,
we can assume that $\theta_0=0$, $r_0=1$ and $\bar{\bx}^0={\bf 0}$.
In addition, as shown in Fig. \ref{vortex3}b, we assume
\be\label{xjt478}
\bx_j(t)=r_j(t)(\cos(\theta_j(t)),\sin(\theta_j(t))^T, \qquad j=1,2,3,
\ee
satisfying $\theta_1^0:=\theta_1(0)<\theta_2^0:=\theta_2(0)<\theta_3^0:=\theta_3(0)<2\pi$ and $\|\bX^0-\tilde \bX^0\|_2\le \delta$ with $0<\delta\le \frac{1}{12}$ sufficiently small and to be
determined later. In fact, from
\bea
\|\bX^0-\tilde \bX^0\|_2^2&=&\sum_{j=1}^3(r_j(0)-1)^2+\sum_{j=1}^{3}4r_j(0)
\sin^2(\varphi_j^0),
\eea
with $\varphi_j^0=\frac{\theta_j^0}{2}-\frac{(j-1)\pi}{3}$ for
$j=1,2,3$,
we can get
\be\label{r1230}
|r_1(0)-1|+|r_2(0)-1|+|r_3(0)-1|\le 3\delta<\frac{1}{2}, \qquad |\theta_1^0|+\left|\theta_2^0-\frac{2\pi}{3}\right|
+\left|\theta_3^0-\frac{4\pi}{3}\right|\le 6\delta<\frac{\pi}{4}.
\ee
Plugging \eqref{xjt478} into \eqref{mccon1}, we get
\bea
&&r_1(t)\cos(\theta_1(t))+r_2(t)\cos(\theta_2(t))+r_3(t)
\cos(\theta_3(t))\equiv 0,\nn \\
&&r_1(t)\sin(\theta_1(t))+r_2(t)\sin(\theta_2(t))
+r_3(t)\sin(\theta_3(t))\equiv 0, \qquad t\ge0.\nonumber
\eea
Solving the above equations, we obtain
\bea\label{orbrr1}
r_2(t)&=&-r_1(t)\frac{\sin(\theta_3(t)-\theta_1(t))}
{\sin(\theta_3(t)-\theta_2(t))}=-r_1(t)\frac{\sin\left(\phi_1(t)+\phi_2(t)\right)}
{\sin\left(\phi_2(t)\right)},\\
\label{orbrr2}
r_3(t)&=&r_1(t)\frac{\sin(\theta_2(t)-\theta_1(t))}
{\sin(\theta_3(t)-\theta_2(t))}=r_1(t)\frac{\sin\left(\phi_1(t)\right)}
{\sin\left(\phi_2(t)\right)},\ \qquad t\ge0,
\eea
where (cf. Fig. \ref{vortex3}b)
\be \label{phi123}
\phi_1(t)=\theta_2(t)-\theta_1(t),\quad \phi_2(t)=\theta_3(t)-\theta_2(t),\qquad
t\ge0.
\ee
By Lemma \ref{lem1}, we have
\begin{equation}\label{orb2}
r_1^2(t)+r_2^2(t)+r_3^2(t)=12t+H_3^0,\qquad t\geq 0,
\end{equation}
with $H_3^0=r_1^2(0)+r_2^2(0)+r_3^2(0)$.
Substituting (\ref{orbrr1}) and (\ref{orbrr2}) into (\ref{orb2}), we can get
\begin{equation}\label{orb3}
r_1(t)=\frac{(12t+H_3^0)^{1/2}\sin\left(\phi_2(t)\right)}
{D^{1/2}(t)}, \quad D(t):=
\sin^2\left(\phi_1(t)\right)+\sin^2\left(\phi_2(t)\right)
+\sin^2\left(\phi_1(t)+\phi_2(t)\right).
\end{equation}
Plugging \eqref{xjt478} into \eqref{GLE} with $N=3$, noting \eqref{orbrr1}-\eqref{phi123} and \eqref{orb3},  we have
\be\label{Phit267}
\dot\Phi(t)=\frac{2}{12t+H_3^0}F(\phi_1(t), \phi_2(t))=
\frac{2}{12t+H_3^0}F(\Phi(t)),\quad t>0,
\ee
where $\Phi(t):=(\phi_1(t),\phi_2(t))^T$ and $F(\Phi)=(f_1(\Phi),f_2(\Phi))^T$ is defined as
\beas
f_1(\Phi)&=&\frac{\sin\left(\phi_1\right)}{\sin\left(\phi_2\right)
\sin\left(\phi_1+\phi_2\right)}\left[\frac{\sin^2\left(\phi_1+\phi_2\right)}{D_{13}(\Phi)}
+\frac{\sin^2\left(\phi_2\right)}{D_{23}(\Phi)}-\frac{\sin^2\left(\phi_2\right)+\sin^2\left(\phi_1+\phi_2\right)}{D_{12}(\Phi)}
\right],\\
f_2(\Phi)&=&\frac{\sin\left(\phi_2\right)}{\sin\left(\phi_1\right)
\sin\left(\phi_1+\phi_2\right)}\left[\frac{\sin^2\left(\phi_1+\phi_2\right)}{D_{13}(\Phi)}
+\frac{\sin^2\left(\phi_1\right)}{D_{12}(\Phi)}-\frac{\sin^2\left(\phi_1\right)+\sin^2\left(\phi_1+\phi_2\right)}{D_{23}(\Phi)}
\right],
\eeas
with
\beas D_{12}(\Phi)&=&\frac{1}{D(\Phi)}\left(\sin^2\left(\phi_2\right)+
\sin^2\left(\phi_1+\phi_2\right)
+2\sin\left(\phi_2\right)\sin\left(\phi_1+\phi_2\right)\cos\left(\phi_1\right)
\right),\\
D_{13}(\Phi)&=&\frac{1}{D(\Phi)}\left(\sin^2\left(\phi_1\right)+
\sin^2\left(\phi_2\right)
-2\sin\left(\phi_1\right)\sin\left(\phi_2\right)
\cos\left(\phi_1+\phi_2\right)\right),\\
D_{23}(\Phi)&=&\frac{1}{D(\Phi)}\left(\sin^2\left(\phi_1\right)+
\sin^2\left(\phi_1+\phi_2\right)
+2\sin\left(\phi_1\right)\sin\left(\phi_1+\phi_2\right)
\cos\left(\phi_2\right)\right),\\
P(\Phi)&=&\frac{1}{D(\Phi)}\left[\sin\left(\phi_2\right)-
\sin\left(\phi_1+\phi_2\right)\cos\left(\phi_1-\frac{2\pi}{3}\right)
+\sin\left(\phi_1\right)\cos\left(\phi_1+\phi_2-\frac{4\pi}{3}\right)\right]^2;
\eeas
and
\be\label{tha187}
\dot\theta_1(t)=\frac{2}{12t+H_3^0}g(\Phi(t)),\qquad t>0,
\ee
with
\be
g(\Phi)=g(\phi_1,\phi_2)=\frac{\sin\left(\phi_1\right)
\sin\left(\phi_1+\phi_2\right)\left(D_{13}(\Phi)-D_{12}(\Phi)\right)}
{\sin\left(\phi_2\right)
D_{12}(\Phi)D_{13}(\Phi)}.\nonumber
\ee

Let
\be\label{st98}
s=\frac{1}{4}\ln \left(\frac{12t+H_3^0}{H_3^0}\right),\qquad
\Psi(s)=\Phi(t)-(2\pi/3,2\pi/3)^T, \qquad s\ge0,
\ee
then \eqref{Phit267} can be re-written as
\be\label{orb4}
\dot\Psi(s)=F\left(\Psi(s)+(2\pi/3,2\pi/3)^T\right)=-2 \Psi(s)+G(\Psi(s)),
\qquad s>0,
\ee
where
\be
G(\Psi)=F\left(\Psi+(2\pi/3,2\pi/3)^T\right)+2\Psi.\nonumber
\ee
It is easy to verify that $\Psi(s)\equiv {\bf 0}$ is
an equilibrium solution of (\ref{orb4}). By the variation-of-constant formula, we have
\be\label{Psis69}
\Psi(s)=e^{-2s}\Psi(0)+\int_0^s e^{-2(s-\tau)}
G(\Psi(\tau))d\tau, \qquad s\ge0.
\ee

By using the Taylor expansion, there exist constants $K_j>0$ ($j=1,2,3$) and $0<\delta_1<1$  such that
\bea\label{bdd367}
&&\|G(\Psi)\|_2\leq \|\Psi\|_2,\quad
\|G(\Psi)\|_2\leq K_1\|\Psi\|_2^2, \quad
|g(\Phi)|=|g(\Psi+(2\pi/3,2\pi/3)^T)|\leq K_2\|\Psi\|_2,\qquad \\
\label{bdd3671}
&&\left|3-P(\Phi)\right|=\left|3-P(\Psi+(2\pi/3,2\pi/3)^T)\right|\leq K_3\|\Psi\|_2^2, \quad \hbox{when}\ \|\Psi\|_2<\delta_1.
\eea
For any $0<\delta_2\le \delta_1$, when $\|\Psi(0)\|_2\le \frac{\delta_2}{2}$
($\Leftrightarrow \|\Phi(0)-(2\pi/3,2\pi/3)^T\|_2\le \frac{\delta_2}{2}$) and
let $S>0$ such that $\|\Psi(s)\|_2\le \delta_2$ for $0\le s\le S$, noting \eqref{Psis69} and \eqref{bdd367}
and using the triangle inequality, we have
\be
\|\Psi(s)\|_2\leq e^{-2s}\|\Psi(0)\|_2+\int_0^s e^{-2(s-\tau)}\| \Psi(\tau)\|_2\,d\tau, \qquad  0\leq s\le S,\nonumber
\ee
which is equivalent to
\be
e^{2s}\|\Psi(s)\|_2=\|\Psi^0\|_2+\int_0^s e^{2\tau}\|\Psi(\tau)\|_2\,d\tau,
\qquad 0\leq s\le S. \nonumber
\ee
Using the  Gronwall's inequality, we get
\be\label{PsisS1}
\|\Psi(s)\|_2\leq \| \Psi(0)\|_2\,e^{-s}\le \| \Psi(0)\|_2\le \frac{\delta_2}{2}, \qquad 0\leq s\le S.
\ee
From \eqref{PsisS1} and using the standard extension theorem for ODEs,
we can obtain
\be\label{PsisS2}
\|\Psi(s)\|_2\leq \| \Psi(0)\|_2\,e^{-s}, \qquad 0\leq s<+\infty.
\ee
Combining \eqref{PsisS2} and \eqref{Psis69}, using the triangle inequality, we obtain
\bea
\|\Psi(s)\|_2&\leq&e^{-2s}\|\Psi(0)\|_2+e^{-2s}\int_0^s e^{2\tau}\| G(\Psi(\tau))\|_2\,d\tau
\leq e^{-2s}\|\Psi(0)\|_2+e^{-2s}\int_0^s e^{2\tau}K_1\|\Psi(\tau)\|_2^2\,d\tau\nn\\
&\leq&e^{-2s}\|\Psi(0)\|_2+e^{-2s}\int_0^s e^{2\tau}K_1\|\Psi(0)\|_2^2e^{-2\tau}\,d\tau
\leq\left[\|\Psi(0)\|_2+K_1\|\Psi(0)\|_2^2\right]e^{-2s}\nn\\
&\le&(1+K_1)\delta_2 e^{-2s}, \qquad 0\leq s<+\infty,\nonumber
\eea
which immediately implies
\begin{equation}
\|\Phi(t)-(2\pi/3,2\pi/3)^T\|_2<(1+K_1)\delta_2\sqrt{\frac{H_3^0}{12t+H_3^0}},
\qquad 0\leq t<+\infty.\nonumber
\end{equation}
Noting \eqref{st98} and \eqref{bdd367}, we have
\be
|g(\Phi)|=|g(\Psi+(\pi/3,\pi/3)^T)|\leq K_2\|\Psi\|_2\le K_2(1+K_1)\delta_2\sqrt{\frac{H_3^0}{12t+H_3^0}},\qquad 0\leq t<+\infty.\nonumber
\ee
This implies that the ODE \eqref{tha187} is globally solvable, and the solution can be written as
\be
\theta_1(t)=\theta_1(0)+\int_0^t\frac{3}{12s+H_3^0}g(\phi_1(s), \phi_2(s))\,ds,\qquad 0\leq t\le +\infty.\nonumber
\ee
Denote $\theta_1^\infty=\lim_{t\to\infty}\theta_1(t)$ and
$\theta(t)=\theta_1(t)-\theta_1^\infty$, then we have
\bea\label{bXt864}
\inf_{r>0} \|\bX(t)-r Q(\theta(t))\tilde \bX(t)\|_2&=&
\inf_{r>0} \left\{12t+H_3^0+3r^2-2r\,d(t)\right\}
=12t+H_3^0-\frac{1}{3}d^2(t)\nn\\
&=&\frac{12t+H_3^0}{3}\left(3-P(\Phi(t))\right), \quad t\ge0,
\eea
where
\beas
d(t):=r_1(t)+r_2(t)\cos(\phi_1(t)-2\pi/3)+r_3(t)\cos(\phi_1(t)+\phi_2(t)-4\pi/3).
\eeas
Noting \eqref{bdd3671}, we have
\be\label{Ph3t1}
\left|3-P(\Phi(t))\right|=\left|3-P(\Psi(t)+(2\pi/3,2\pi/3)^T)\right|\leq K_3\|\Psi(t)\|_2^2\le \frac{K_3(1+K_1)^2\delta_2H_3^0}{12t+H_3^0},\ \ t\ge0.
\ee
For any $\varepsilon>0$, taking $\delta_3=\frac{\varepsilon}{2K_3(1+K_1)^2H_3^0}$ and
$0<\delta=\min\left\{\frac{1}{12},\frac{\delta_1}{12},\delta_3\right\}$,
 when $\|\bX^0-\tilde\bX^0\|_2< \delta$, noting \eqref{Ph3t1} and \eqref{bXt864}, we get
\be
\sup_{t\geq 0}\ \inf_{r>0} \|\bX(t)-r Q(\theta(t))\tilde \bX(t)\|_2
\leq K_3(1+K_1)^2H_3^0\delta< \varepsilon,\label{fin}
\ee
which completes the proof by taking $\delta_2=\delta$ in the above proof.
\end{proof}

\subsection{Collision patterns in the case with opposite winding numbers}
Without loss of generality, we assume $m_1=m_3=+1$ and $m_2=-1$ in
\eqref{GLE} with $N=3$. Then we have $M_0=\frac{1}{2}[(N^+-N^-)^2-N]=
\frac{1}{2}(1^2-3)=-1<0$, thus finite time collision must happen.

\begin{theorem} For any given initial data $\bX^0\in {\mathbb R}_*^{2\times 3}$ in \eqref{GLEi} with $N=3$,
we have

(i) If $|\mathbf
x_1^0-\mathbf x_2^0|=|\mathbf x_2^0-\mathbf x_3^0|$, then the three
vortices be a collision cluster and they will collide at $\bar\bx^0$ when
$t\to T_{\rm max}^-=\frac{H_1^0}{12}$ with $H_1^0=\sum_{1\leq j<l\leq 3}|\mathbf x_j^0-\mathbf
x_l^0|^2$.

(ii) If $|\mathbf x_1^0-\mathbf x_2^0|< |\mathbf
x_2^0-\mathbf x_3^0|$, then only $\mathbf x_1$ and $\mathbf x_2$ form
a collision cluster, and respectively, if $|\mathbf x_1^0-\mathbf x_2^0|> |\mathbf x_2^0-\mathbf x_3^0|$, then
only $\mathbf x_2$ and $\mathbf x_3$ form a collision cluster.
Moreover, the collision time $0<T_{\rm max}<\frac{H_1^0}{12}$.
\end{theorem}

\begin{proof} 

(i) If the initial data $\bX^0\in {\mathbb R}_*^{2\times 3}$ in (\ref{GLEi}) with $N=3$ is collinear, i.e. there exists a unit vector
${\bf e}\in {\mathbb R}^2$  such that
\be
\bx_j^0=\bx^0_2+a_j^0 {\bf e}, \qquad j=1,2,3,\nonumber
\ee
satisfying $a_2^0=0$, $a_1^0< a_3^0$ and $0<|a_1^0|\le |a_3^0|$ (without loss of generality, otherwise we need only switch $\bx_1$ and $\bx_3$).
Based on the results in Lemma \ref{lmcolin} and Theorem \ref{NNC},
we know that there exist $a_j(t)$ ($j=1,2,3$) such that
\be
\bx_j(t)=\bx^0_2+a_j(t) {\bf e}, \qquad j=1,2,3,\nonumber
\ee
satisfying $a_1(t)+a_2(t)+a_3(t)\equiv a_1^0+a_2^0+a_3^0$ and $a_1(t)<a_3(t)$ for $0\le t<T_{\rm max}$ and
\bea
\dot a_1(t)&=&-\frac{2}{a_1(t)-a_2(t)}+\frac{2}{a_1(t)-a_3(t)}=
\frac{2[a_3(t)-a_2(t)]}{[a_2(t)-a_1(t)][a_3(t)-a_1(t)]}, \nonumber\\
\dot a_2(t)&=&-\frac{2}{a_2(t)-a_1(t)}-\frac{2}{a_2(t)-a_3(t)}
=\frac{2[3a_2(t)-(a_1^0+a_2^0+a_3^0)]}{[a_2(t)-a_1(t)][a_3(t)-a_2(t)]},\qquad t>0,\nonumber\\
\dot a_3(t)&=&\frac{2}{a_3(t)-a_1(t)}-\frac{2}{a_3(t)-a_2(t)}
=\frac{2[a_1(t)-a_2(t)]}{[a_3(t)-a_1(t)][a_3(t)-a_2(t)]},\nonumber
\eea
with the initial data \eqref{init3478}.

If $|a_1^0|=|a_3^0|$, i.e.
$a_3^0=-a_1^0>0$, then the above ODEs with \eqref{init3478} admits the unique
solution as
\be
a_1(t)=-\sqrt{(a_1^0)^2-2t}, \quad a_2(t)\equiv 0, \quad a_3(t)=\sqrt{(a_3^0)^2-2t},
\qquad 0\le t<T_{\rm max}:=\frac{1}{2}(a_3^0)^2,\nonumber
\ee
which immediately implies that the three
vortices be a collision cluster and they  collide at $\bar\bx^0=\bx_2^0$ when
$t\to T_{\rm max}^-=\frac{H_1^0}{12}$ with $H_1^0=\sum_{1\leq j<l\leq 3}|\mathbf x_j^0-\mathbf
x_l^0|^2=6(a_3^0)^2$.

If $|a_1^0|<|a_3^0|$, then $a_3^0>0$. If $0=a_2^0<a_1^0<a_3^0$,
then we can show that $a_2(t)<a_1(t)<a_3(t)$ for $0\le t< T_{\rm max}$
and $a_1(t)$, $a_2(t)$ and $a_3(t)$ are monotonically decreasing,
increasing and increasing functions over $t\in [0,T_{\rm max})$, respectively. Thus only $\mathbf x_1$ and $\mathbf x_2$ form
a collision cluster among the $3$ vortices.
On the other hand, if $a_1^0<0=a_2^0<a_3^0$, then we can show that $a_1(t)<a_2(t)<a_3(t)$ for $0\le t< T_{\rm max}$
and $a_1(t)$ and $a_2(t)$ are monotonically increasing
and decreasing functions over $t\in [0,T_{\rm max})$, respectively.
In addition we have $a_1(T_{\rm max})\le a_2(T_{\rm max})<0$ and
$a_3(T_{\rm max})=a_3^0+a_2^0+a_1^0-a_1(T_{\rm max})-a_2(T_{\rm max})>0$,
therefore, again only $\mathbf x_1$ and $\mathbf x_2$ form
a collision cluster among the $3$ vortices.

(ii) If the initial data $\bX^0\in {\mathbb R}_*^{2\times 3}$ in (\ref{GLEi}) with $N=3$ is not collinear, i.e. the initial locations of the $3$ vortices form a triangle.
Without loss of generality, as shown in Fig. \ref{vortex3}c, we assume $\bar\bx^0={\bf 0}$ and $d_{12}^0:=d_{12}(0)\le d_{23}^0:=d_{23}(0)$.
Thus $0<\theta_3^0:=\theta_3(0)\le \theta_1^0:=\theta_1(0)<\pi$ satisfying
$\theta_1^0+\theta_2^0+\theta_3^0=\pi$ (cf. Fig. \ref{vortex3}b).
From \eqref{GLE} with $N=3$, we get
\bea\label{d1251}
\dot{d}_{12}(t)&=&-\frac{4}{d_{12}(t)}+\frac{2\cos(\theta_{1}(t))}{d_{13}(t)}
-\frac{2\cos(\theta_2(t))}{d_{23}(t)},\
\dot\theta_{3}(t)=-B(t)\left[d_{13}^2(t)+d_{23}^2(t)\right]<0,\\
\label{d1252}
\dot{d}_{13}(t)&=&\frac{4}{d_{13}(t)}-\frac{2\cos(\theta_{1}(t))}{d_{12}(t)}
-\frac{2\cos(\theta_3(t))}{d_{23}(t)},\
\dot\theta_{2}(t)=B(t)\left[d_{12}^2(t)+d_{23}^2(t)+2d_{13}^2(t)\right]>0,\qquad\\
\label{d1253}
\dot{d}_{23}(t)&=&-\frac{4}{d_{23}(t)}-\frac{2\cos(\theta_{2}(t))}{d_{12}(t)}
+\frac{2\cos(\theta_3(t))}{d_{13}(t)},\
\dot\theta_{1}(t)=-B(t)\left[d_{12}^2(t)+d_{13}^2(t)\right]<0.
\eea

If $d_{12}^0=d_{23}^0$, then $0<\theta_3^0=\theta_1^0<\frac{\pi}{2}$ (cf.
Fig. \ref{vortex3}c),
this together with \eqref{d1251}-\eqref{d1253} yields
\be
d_{12}(t)=d_{23}(t), \ 0<\theta_3(t)=\theta_1(t)<\frac{\pi}{2}, \ 0\le t<T_{\rm max}; \quad \lim_{t\to T_{\rm max}^-}\theta_3(t)=
\lim_{t\to T_{\rm max}^-}\theta_1(t)=0, \nn
\ee
which immediately implies that the three vortices are forming
a collision cluster. By using Theorem \ref{fi},
we get $T_{\rm max}=H^0_1/12$.

If $0<d_{12}^0<d_{23}^0$, then $0<\theta_3^0<\theta_1^0<\frac{\pi}{2}$ (cf.
Fig. \ref{vortex3}b). From \eqref{d1251} and \eqref{d1253}, we have
\bea
\dot{d}_{23}(t)-\dot{d}_{12}(t)&=&\frac{[4-2\cos(\theta_{2}(t))](d_{23}(t)
-d_{12}(t))}{d_{12}(t)d_{23}(t)}
+\frac{2[\cos(\theta_{3}(t))-\cos
(\theta_{1}(t))]}{d_{13}(t)}>0,\qquad \nn \\
\dot\theta_{1}(t)-\dot\theta_{3}(t)&=&B(t)\left[d_{23}^2(t)-d_{12}^2(t)\right]>0,
\qquad t>0.\nn
\eea
Then we have
\be
d_{23}(t)\ge d_{12}(t)+d_{23}^0-d_{12}^0,\quad
\theta_{1}(t)\ge \theta_{3}(t)+\theta_{1}^0-\theta_{3}^0, \qquad
0\le t\le T_{\rm max}.\nn
\ee
This, together with that $\theta_1(t)$ and $\theta_3(t)$ are monotonically decreasing functions, $0<\theta_2(t)=\pi-\theta_1(t)-\theta_3(t)<\pi$
is a monotonically increasing functions
and finite time collision must happen (cf.
Fig. \ref{vortex3}c),
we get that $\lim_{t\to T_{\rm max}^-}d_{12}(t)=0$ and
$\lim_{t\to T_{\rm max}^-}\theta_3(t)=0$. Thus only the two vortices
$\mathbf x_1(t)$ and $\mathbf x_2(t)$ form a collision cluster among the
$3$ vortices. By using
Theorem \ref{fi}, we get the collision time $0<T_{\rm max}< \frac{H_1^0}{12}$.
\end{proof}

\section{Analytical solutions under special initial setups}
\setcounter{equation}{0}
\setcounter{figure}{0}
Let $0\le \theta_0<2\pi$ be a constant, $n\ge2$ be an integer,
$0<a_1<a_2$ be two constants, $C_1:=\frac{1}{2}\left(a_1^2+a_2^2\right)$, $C_2:=\frac{1}{2}\left(a_2^2-a_1^2\right)$, and $m_0=+1$ or $-1$.
Denote
\be
\theta_n^j=\frac{2(j-1)\pi}{n}+\theta_0, \qquad \alpha_n^j=\frac{2(j-1)\pi}{n}+\frac{\pi}{n}+\theta_0, \qquad 1\le j\le n. \nonumber
\ee

\subsection{For the interaction of two clusters}
Here we take $N=2n$ with $n\ge2$.

\bpro\label{lem4}
Taking $m_j=m_0$  for $1\le j\le N=2n$ and
the initial data $\bX^0$ in \eqref{GLEi} as
\be\label{patI11}
\mathbf x_j^0=a_1\left(\cos(\theta_n^j),\sin(\theta_n^j)\right)^T,\quad
\mathbf x_{n+j}^0=a_2\left(\cos(\theta_n^j),\sin(\theta_n^j)\right)^T,
\quad 1\le j\le n,
\ee
then the solution of the ODEs \eqref{GLE} with \eqref{patI11} can be given as
\be\label{solgle11}
\mathbf x_j(t)=\sqrt{\rho_1}(t)\left(\cos(\theta_n^j),\sin(\theta_n^j)\right)^T,\
\mathbf x_{n+j}(t)=\sqrt{\rho_2(t)}\left(\cos(\theta_n^j),\sin(\theta_n^j)\right)^T,
\  1\le j\le n,\   t\ge0,
\ee
where when $n=2$,
\be\label{rho12560}
\rho_1(t)=C_1+6t-
\sqrt{C_2^2+8C_1t+24t^2}, \quad
\rho_2(t)=C_1+6t+
\sqrt{C_2^2+8C_1t+24t^2}, \quad
t\ge0;
\ee
and when $n\ge3$,
\be\label{rho1256}
\rho_1(t)\sim \alpha_1 t,  \qquad \rho_2(t)\sim \beta_1 t, \qquad
t\gg 1,
\ee
with $\alpha_1$ and $\beta_1$ being two positive constants
satisfying
\be \label{apb11}
0<\alpha_1<\beta_1, \qquad \alpha_1+\beta_1=8n-4,\qquad
\beta_1-\alpha_1
=4n\frac{\beta_1^{n/2}+\alpha_1^{n/2}}{\beta_1^{n/2}-\alpha_1^{n/2}}.
\ee
Specifically, when $n\gg1$, we have
\be\label{apb1112}
\alpha_1\approx 2n-2, \qquad \beta_1\approx 6n-2.
\ee
\epro

\begin{proof} Noting the symmetry of the
ODEs \eqref{GLE} with the initial data (\ref{patI11}),
we can take the solution ansatz \eqref{solgle11}.
Substituting \eqref{solgle11} into \eqref{GLE} and \eqref{GLEi},
we obtain
\bea\label{r1}
\dot{\rho}_1(t)&=&4\sum_{j=2}^n \frac{\bnn(\theta_n^1)\cdot\left(\bnn(\theta_n^1)-
\bnn(\theta_n^j)\right)}{\left|\bnn(\theta_n^1)-\bnn(\theta_n^j)\right|^2}
+4\sum_{j=1}^n\frac{\bnn(\theta_n^1)\cdot\left(\rho_1(t)\bnn(\theta_n^1)-
\sqrt{\rho_1(t)\rho_2(t)}\bnn(\theta_n^j)\right)}{\left|\sqrt{\rho_1(t)}\bnn(\theta_n^1)-
\sqrt{\rho_2(t)}\bnn(\theta_n^j)\right|^2}\nn\\
&=&2n-2+4\sum_{j=1}^n\frac{\rho_1(t)-\sqrt{\rho_1(t)\rho_2(t)}
\cos(\theta_n^1-\theta_n^j)}{\rho_1(t)+\rho_2(t)-2\sqrt{\rho_1(t)\rho_2(t)}
\cos(\theta_n^1-\theta_n^j)},\qquad t>0,
\eea
\bea\label{r2}
\dot{\rho}_2(t)&=&4\sum_{j=1}^n\frac{\bnn(\theta_n^1)\cdot\left(\rho_2(t)\bnn(\theta_n^1)-
\sqrt{\rho_1(t)\rho_2(t)}\bnn(\theta_n^j)\right)}{\left|\sqrt{\rho_2(t)}\bnn(\theta_n^1)-
\sqrt{\rho_1(t)}\bnn(\theta_n^j)\right|^2}+4\sum_{j=2}^n \frac{\bnn(\theta_n^1)\cdot\left(\bnn(\theta_n^1)-
\bnn(\theta_n^j)\right)}{\left|\bnn(\theta_n^1)-\bnn(\theta_n^j)\right|^2}
\nn\\
&=&2n-2+4\sum_{l=1}^n\frac{\rho_2(t)-\sqrt{\rho_1(t)\rho_2(t)}
\cos(\theta_n^1-\theta_n^j)}{\rho_1(t)+\rho_2(t)-2\sqrt{\rho_1(t)\rho_2(t)}
\cos(\theta_n^1-\theta_n^j)},\qquad t>0,
\eea
where
\be\label{ntheta}
\bnn(\theta)=(\cos(\theta),\sin(\theta))^T, \qquad \theta\in {\mathbb R}.
\ee
Summing \eqref{r1} and \eqref{r2}, we have
\begin{eqnarray}\label{tp31}
\dot\rho_1(t)+\dot\rho_2(t)=8n-4,\qquad \qquad t>0,
\end{eqnarray}
Subtracting \eqref{r1} from \eqref{r2}, we get
\begin{eqnarray}\label{tp3101}
\dot\rho_2(t)-\dot \rho_1(t)=4n\frac{\rho_2^{n/2}(t)+\rho_1^{n/2}(t)}{\rho_2^{n/2}(t)-\rho_1^{n/2}(t)}=
4n+\frac{8n\rho_1^{n/2}(t)}{\rho_2^{n/2}(t)-\rho_1^{n/2}(t)}, \qquad t>0.
\end{eqnarray}
Here we use the equality
\[\sum_{j=1}^n\frac{x^2-1}{x^2+1-2x
\cos(\theta_n^1-\theta_n^j)}=n\frac{x^n+1}{x^n-1}, \qquad 1<x\in {\mathbb R}.
\]
Combining \eqref{tp31} and \eqref{tp3101}, we obtain
\be\label{rho678}
\dot\rho_1(t)=2n-2-\frac{4n\rho_1^{n/2}(t)}{\rho_2^{n/2}(t)-\rho_1^{n/2}(t)}, \quad \dot\rho_2(t)=6n-2+\frac{4n\rho_1^{n/2}(t)}{\rho_2^{n/2}(t)-\rho_1^{n/2}(t)}, \quad t\ge0,
\ee
with the initial data
\be\label{initrh120}
\rho_1(0)=\rho_1^0:=a_1^2<\rho_2(0)=\rho_2^0:=a_2^2.
\ee
When $n=2$, we can solve \eqref{rho678} with \eqref{initrh120} analytically
and obtain the solution  \eqref{rho12560} immediately.
When $n\ge3$, noting that all the vortices have
the same winding number, i.e. $T_{\rm max}=+\infty$ by using Theorem \ref{NNC}, we get $\rho_1(t)<\rho_2(t)$ for $t\ge0$ and thus
\be
\dot\rho_2(t)>0,\qquad \dot\rho_2(t)-\dot \rho_1(t)>0,\qquad t\ge0.\nonumber
\ee
Therefore, we conclude that $\rho_2(t)$ and
$\rho_2(t)-\rho_1(t)$  are monotonically increasing
functions when $t\ge0$ and $\displaystyle\lim_{t\rightarrow
+\infty}\rho_2(t)=+\infty$ by noting Theorem \ref{fi}.
From \eqref{rho678}, we can conclude that there exist two positive constants $0<\alpha_1<\beta_1$ such that \eqref{rho1256} is valid.
Plugging \eqref{rho1256} into \eqref{tp31}, we get
\eqref{apb11} immediately. When $n\gg1$, \eqref{apb11} yields
\be
 \alpha_1+\beta_1=8n-4,\qquad
\beta_1-\alpha_1
\approx 4n, \nonumber
\ee
which immediately implies \eqref{apb1112}. In addition,
Figure \ref{rho121} depicts the solution
$\rho_1(t)$ and $\rho_2(t)$ of \eqref{rho678} obtained numerically with $\rho_1(0)=1$ and $\rho_2(0)=4$ for different $n\ge2$.
\end{proof}

\begin{figure}[t!]
\centerline{\psfig{figure=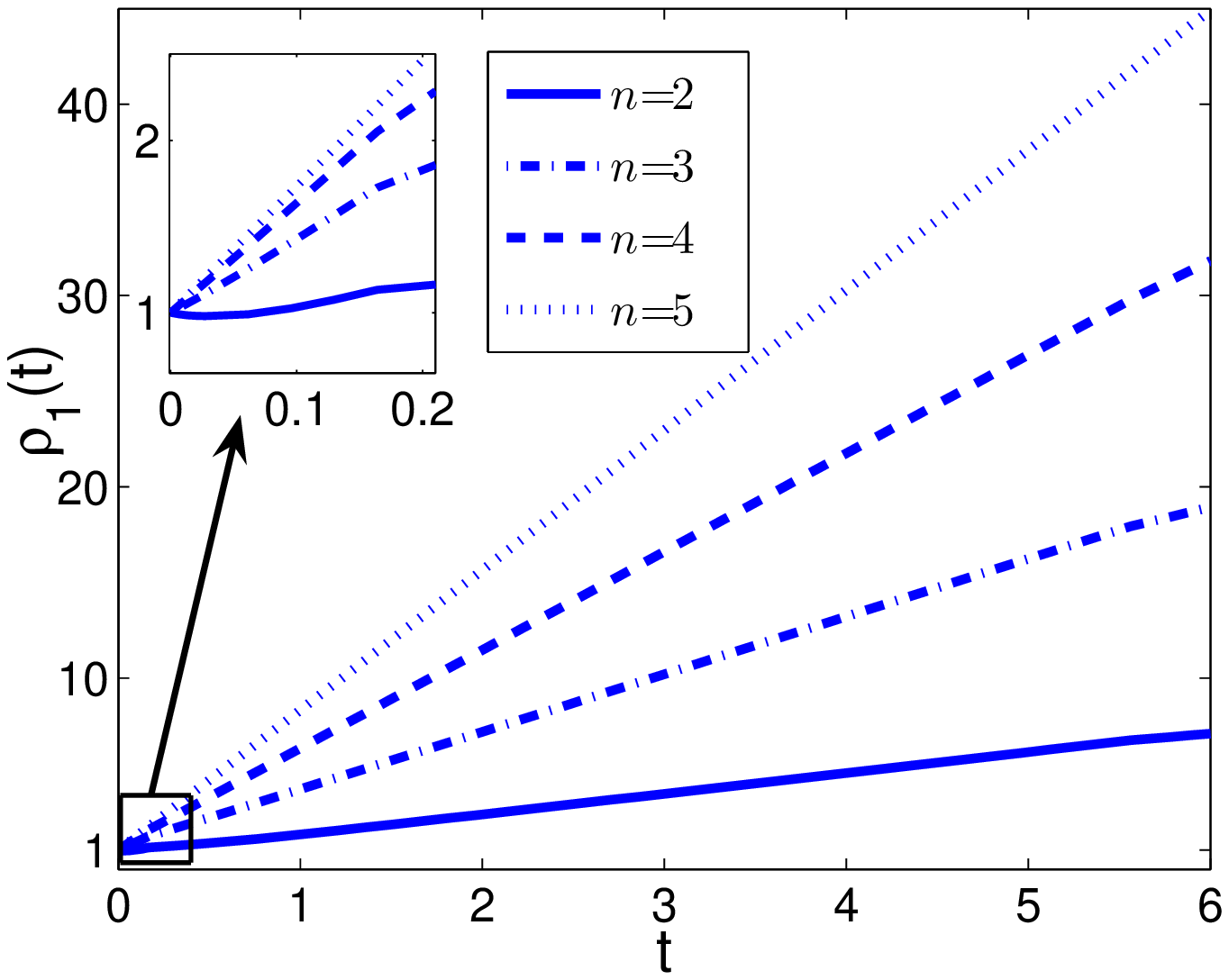,height=5cm,width=6.5cm,angle=0}
\psfig{figure=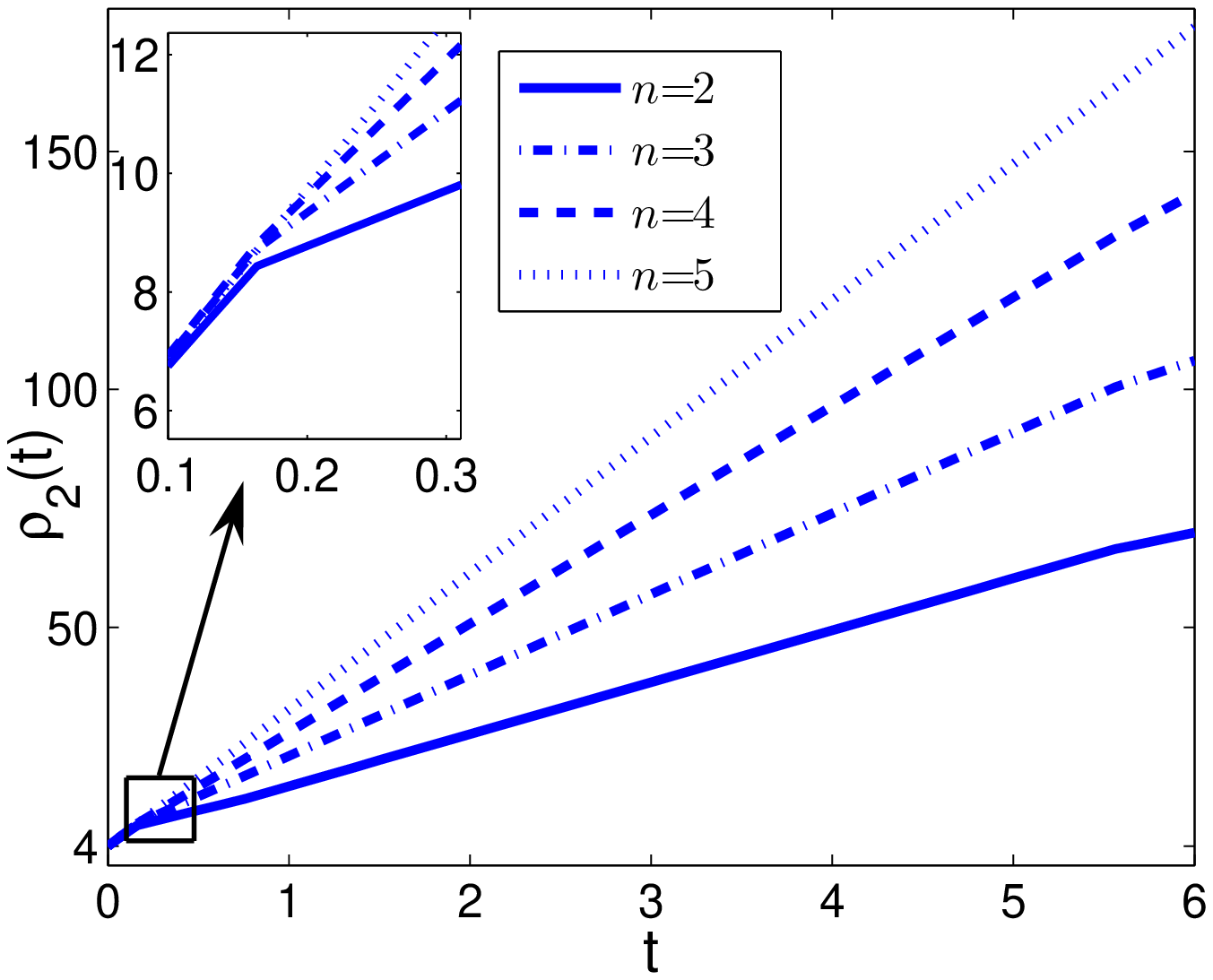,height=5cm,width=6.5cm,angle=0}}
\caption{Time evolution of $\rho_1(t)$ (left) and $\rho_2(t)$ (right)
of \eqref{rho678} with $\rho_1^0=1$ and $\rho_2^0=4$ for different $n\ge2$.}
\label{rho121}
\end{figure}

\begin{proposition}\label{lem7}
Taking $m_j=m_0$  for $1\le j\le N=2n$ and
the initial data $\bX^0$ in \eqref{GLEi} as
\be\label{patI12}
\mathbf x_j^0=a_1\left(\cos(\theta_n^j),\sin(\theta_n^j)\right)^T,\quad
\mathbf x_{n+j}^0=a_2\left(\cos(\alpha_n^j),\sin(\alpha_n^j)\right)^T,
\quad 1\le j\le n,
\ee
then the solution of the ODEs \eqref{GLE} with \eqref{patI12} can be given as
\be\label{solgle12}
\mathbf x_j(t)=\sqrt{\rho_1(t)}\left(\cos(\theta_n^j),\sin(\theta_n^j)\right)^T,\
\mathbf x_{n+j}(t)=\sqrt{\rho_2(t)}\left(\cos(\alpha_n^j),\sin(\alpha_n^j)\right)^T,
\   1\le j\le n,\   t\ge0,
\ee
where when $n=2$,
\be
\rho_1(t)=C_1+6t-C_2\left(
1+\frac{6t}{C_1}\right)^{2/3},\quad \rho_2(t)=C_1+6t+C_2\left(
1+\frac{6t}{C_1}\right)^{2/3}, \qquad t\ge0; \nonumber
\ee
and when $n\ge3$,
\be
\rho_1(t)\sim \alpha_2 t, \qquad \rho_2(t)\sim \beta_2 t, \qquad t\gg1, \nonumber
\ee
with $\alpha_2$ and $\beta_2$ being two positive constants satisfying
\begin{equation*}
0<\alpha_2< \beta_2, \quad
\alpha_2+\beta_2=8n-4,\quad
\beta_2-\alpha_2=4n\frac{\beta_2^{n/2}-\alpha_2^{n/2}}
{\beta_2^{n/2}+\alpha_2^{n/2}}.
\end{equation*}
Specifically,  when $n\gg1$, $\alpha_2\approx 2n-2$ and $\beta_2\approx 6n-2$.
\end{proposition}

\begin{proof}
The proof is analogue to that of Proposition \ref{lem4} and thus it is
omitted here for brevity.
\end{proof}

\begin{proposition}\label{th31}
Taking $m_j=m_0$ and $m_{n+j}=-m_0$ for $1\le j\le n$ and
the initial data $\bX^0$ in \eqref{GLEi} as
\eqref{patI12},
then the solution of the ODEs \eqref{GLE} with \eqref{patI12} can be given as
\eqref{solgle12},
where
\be
\rho_1(t)>0, \quad \rho_2(t)>0, \quad 0\le t<T_c:=\frac{1}{4}(a_1^2+a_2^2), \qquad \lim_{t\to T_c^-}\rho_1(t)=\lim_{t\to T_c^-}\rho_2(t)=0, \nonumber
\ee
which implies that the $N=2n$ vortices will be a (finite time) collision cluster.
\end{proposition}

\begin{proof} Similar to the proof of
Proposition \ref{lem4}, noting the symmetry of the
ODEs \eqref{GLE} with the initial data (\ref{patI12}),
we can take the solution ansatz \eqref{solgle12}.
In addition, plugging \eqref{solgle12} into \eqref{GLE} and \eqref{GLEi},
we get
\bea\label{r113}
\dot{\rho}_1(t)&=&2n-2-4\sum_{l=1}^n\frac{\rho_1(t)-
\sqrt{\rho_1(t)\rho_2(t)}\cos(\theta_n^1-\alpha_n^{l})}
{\rho_1(t)+\rho_2(t)-2\sqrt{\rho_1(t)\rho_2(t)}\cos(\theta_n^1-\alpha_n^{l})},\\
\label{r213}
\dot{\rho}_2(t)&=&2n-2-4\sum_{l=1}^n
\frac{\rho_2(t)-\sqrt{\rho_1(t)\rho_2(t)}\cos(\alpha_n^{1}-\theta_n^l)}
{\rho_1(t)+\rho_2(t)-2\sqrt{\rho_1(t)\rho_2(t)}\cos(\alpha_n^{1}-\theta_n^l)},
\eea
with the initial data \eqref{initrh120}.


Summing \eqref{r113} and \eqref{r213}, we obtain
\be\label{rh12789}
\dot\rho_1(t)+\dot\rho_2(t)=4n-4
-4\sum_{l=1}^{n}1=4n-4-4n=-4,\qquad t\ge0.
\ee
Subtracting \eqref{r113} from \eqref{r213}, we get

\begin{eqnarray}\label{new1}
\dot\rho_2(t)-\dot \rho_1(t)=-4n\frac{\rho_2^{n/2}(t)-\rho_1^{n/2}(t)}{\rho_2^{n/2}(t)+\rho_1^{n/2}(t)}=
-4n+\frac{8n\rho_1^{n/2}(t)}{\rho_2^{n/2}(t)+\rho_1^{n/2}(t)}, \qquad t>0.
\end{eqnarray}
Here we use the equality
\[\sum_{j=1}^n\frac{x^2-1}{x^2+1-2x
\cos(\theta_n^1-\theta_n^j+\frac{\pi}{n})}=n\frac{x^n-1}{x^n+1}, \qquad 1<x\in {\mathbb R}.
\]
Combining \eqref{rh12789} and \eqref{new1}, we obtain
\be\label{rho431}
\dot\rho_1(t)=2n-2-\frac{4n\rho_1^{n/2}(t)}{\rho_2^{n/2}(t)+\rho_1^{n/2}(t)}, \quad \dot\rho_2(t)=-2n-2+\frac{4n\rho_1^{n/2}(t)}{\rho_2^{n/2}(t)+\rho_1^{n/2}(t)}, \quad t\ge0,
\ee
with the initial data \eqref{initrh120}.

Solving \eqref{rh12789} by noting \eqref{initrh120}, we get
\begin{equation*}\label{p11}
\rho_1(t)+\rho_2(t)=-4t+a_1^2+a_2^2, \qquad 0\le t<T_c:=\frac{1}{4}(a_1^2+a_2^2).
\end{equation*}
Noticing $N^+=N^-=n=\frac{N}{2}$, thus $M_0=-\frac{N}{2}=-n<0$ by noting
\eqref{mjml}. From Theorem \ref{fi}, finite time collision
must happen among the $N=2n$ vortices. Thus
there exist $1\leq j_0\leq n$ and $1\leq l_0\leq
n$ such that the vortex dipole $\mathbf x_{j_0}$ and $\mathbf
x_{n+l_0}$ will collide at $t=T_c$, i.e. $\rho_1(T_c)=\rho_2(T_c)=0$.
Therefore, the $N=2n$ vortices will be a (finite time) collision cluster.
In addition,
Figure \ref{rho122} depicts the solution
$\rho_1(t)$ and $\rho_2(t)$ of \eqref{rho431} obtained numerically with $\rho_1(0)=1$ and $\rho_2(0)=4$ for different $n\ge2$.
\end{proof}

\begin{remark}
When $a_1=a_2$, i.e. $\rho_1^0=\rho_2^0$, we can get
\begin{equation*}
\rho_1(t)=\rho_2(t)=-2t+a_1^2,
\end{equation*}
which also implies the  $N=2n$ vortices will be a (finite time) collision cluster.
\end{remark}

\begin{figure}[t!]
\centerline{\psfig{figure=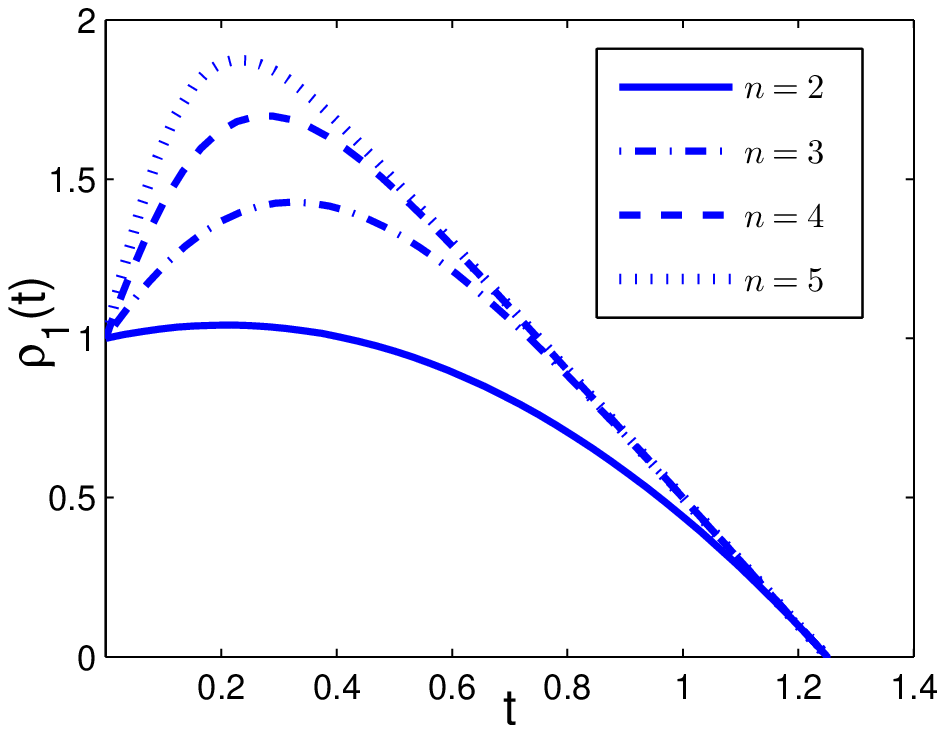,height=5cm,width=6.5cm,angle=0}
\psfig{figure=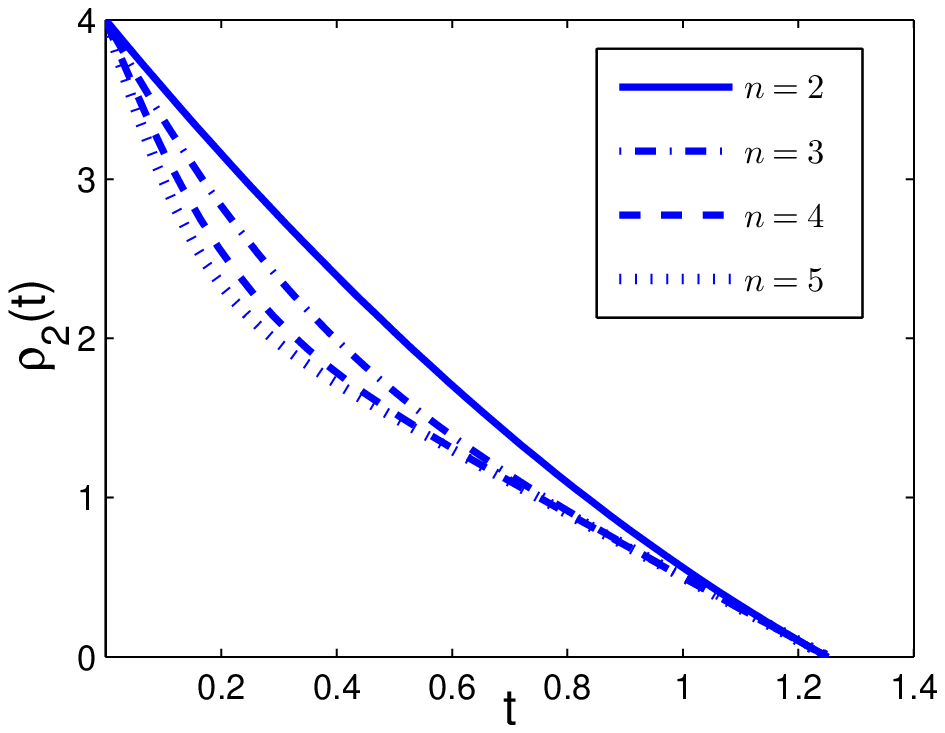,height=5cm,width=6.5cm,angle=0}}
\caption{Time evolution of $\rho_1(t)$ (left) and $\rho_2(t)$ (right)
of \eqref{rho431} with $\rho_1^0=1$ and $\rho_2^0=4$ for different $n\ge2$.}
\label{rho122}
\end{figure}

\subsection{For the interaction of two clusters and a single vortex}
Here we take $N=2n+1$ with $n\ge2$.

\begin{proposition}\label{lem5}
Taking $m_j=m_0$  for $1\le j\le N=2n+1$ and
the initial data $\bX^0$ in \eqref{GLEi} as
\be\label{patI21}
\bx_{N}^0={\bf 0}, \quad \mathbf x_j^0=a_1\left(\cos(\theta_n^j),\sin(\theta_n^j)\right)^T,\quad
\mathbf x_{n+j}^0=a_2\left(\cos(\theta_n^j),\sin(\theta_n^j)\right)^T,
\quad 1\le j\le n,
\ee
then the solution of the ODEs \eqref{GLE} with \eqref{patI21} can be given as
\be\label{solgle21}
\bx_{N}(t)\equiv {\bf 0},\ \mathbf x_j(t)=\sqrt{\rho_1}(t)\left(\cos(\theta_n^j),\sin(\theta_n^j)\right)^T,\
\mathbf x_{n+j}(t)=\sqrt{\rho_2(t)}\left(\cos(\theta_n^j),\sin(\theta_n^j)\right)^T,
\  1\le j\le n,
\ee
where when $n=2$,
\be
\rho_1(t)=C_1+10t-
\sqrt{C_2^2+8C_1t+40t^2}, \quad
\rho_2(t)=C_1+10t+
\sqrt{C_2^2+8C_1t+40t^2}, \quad
t\ge0; \nonumber
\ee
and when $n\ge3$,
\be
\rho_1(t)\sim \alpha_3 t,  \qquad \rho_2(t)\sim \beta_3 t, \qquad
t\gg 1, \nonumber
\ee
with $\alpha_3$ and $\beta_3$ being two positive constants
satisfying
\be
0<\alpha_3<\beta_3, \quad \alpha_3+\beta_3=8n+4,\quad
\beta_3-\alpha_3=4n\frac{\beta_3^{n/2}+\alpha_3^{n/2}}
{\beta_3^{n/2}-\alpha_3^{n/2}}. \nonumber
\ee
Specifically, when $n\gg1$, we have $\alpha_3\approx 2n+2$, and $\beta_3\approx 6n+2.$
\end{proposition}

\begin{proof}
Due to symmetry, we get $\bx_N(t)\equiv {\bf 0}$ for $t\ge0$.
The rest of the proof is analogue to that of Proposition \ref{lem4} and thus it is
omitted here for brevity.
\end{proof}

\begin{proposition}\label{lem8}
Taking $m_j=m_0$  for $1\le j\le N=2n+1$ and
the initial data $\bX^0$ in \eqref{GLEi} as
\be\label{patI22}
\bx_N^0={\bf 0}, \quad \mathbf x_j^0=a_1\left(\cos(\theta_n^j),\sin(\theta_n^j)\right)^T,\quad
\mathbf x_{n+j}^0=a_2\left(\cos(\alpha_n^j),\sin(\alpha_n^j)\right)^T,
\quad 1\le j\le n,
\ee
then the solution of the ODEs \eqref{GLE} with \eqref{patI22} can be given as
\be\label{solgle22}
\bx_N(t)\equiv {\bf 0},\ \mathbf x_j(t)=\sqrt{\rho_1(t)}\left(\cos(\theta_n^j),\sin(\theta_n^j)\right)^T,\
\mathbf x_{n+j}(t)=\sqrt{\rho_2(t)}\left(\cos(\alpha_n^j),\sin(\alpha_n^j)\right)^T,
\  1\le j\le n,
\ee
where when $n=2$,
\be
\rho_1(t)=C_1+10t-C_2\left(
1+\frac{10t}{C_1}\right)^{2/5},\quad \rho_2(t)=C_1+10t+C_2\left(
1+\frac{10t}{C_1}\right)^{2/5},\qquad
t\ge0; \nonumber
\ee
and when $n\ge3$,
\be
\rho_1(t)\sim \alpha_4 t,  \qquad \rho_2(t)\sim \beta_4 t, \qquad
t\gg 1, \nonumber
\ee
with $\alpha_4$ and $\beta_4$ being two positive constants
satisfying
\be \label{apb22}
0<\alpha_4< \beta_4, \quad \alpha_4+\beta_4=8n+4,\quad
\beta_4-\alpha_4=4n\frac{\beta_4^{n/2}-\alpha_4^{n/2}}
{\beta_4^{n/2}+\alpha_4^{n/2}}. \nonumber
\ee
Specifically, when $n\gg1$, $\alpha_4\approx 2n+2$ and  $\beta_4\approx 6n+2$.
\end{proposition}

\begin{proof}
Due to symmetry, we get $\bx_N(t)\equiv {\bf 0}$ for $t\ge0$, and the rest of the proof is analogue to that of Proposition \ref{lem4} and \ref{th31}, thus it is
omitted here for brevity.
\end{proof}


\begin{proposition}\label{lem6}
Taking $m_N=-m_0$, $m_j=m_0$  for $1\le j\le 2n=N-1$ and
the initial data $\bX^0$ in \eqref{GLEi} as
\eqref{patI21}, then the solution of the ODEs
\eqref{GLE} with \eqref{patI21} can be given as
 \eqref{solgle21},
where

(i) when $n=2$, then $\mathbf x_1(t)$, $\mathbf x_2(t)$ and $\mathbf x_5(t)$
be a collision cluster among the $5$ vortices and they will
collide at the origin $(0,0)^T$ in finite time;

(ii) when $n=3$, then
\be\label{p516}
\rho_1(t)\sim \left(\frac{a_1a_2}{a_1+a_2}\right)^2,
\qquad \rho_2(t)\sim 12t, \qquad t\gg1;
\ee

(iii) when $n\geq 4$, then
\be
\rho_1(t)\sim \alpha_5 t,\qquad
\rho_2(t)\sim \beta_5 t, \qquad t\gg 1, \nonumber
\ee
with $\alpha_5$ and $\beta_5$ being two positive constants satisfying
\be
0<\alpha_5< \beta_5, \quad \alpha_5+\beta_5=8n-12,\quad
\beta_5-\alpha_5=4n\frac{\beta_5^{n/2}+\alpha_5^{n/2}}
{\beta_5^{n/2}-\alpha_5^{n/2}}. \nonumber
\ee
Specifically, when $n\gg1$, we have $\alpha_5\approx 2n-6$, and $\beta_5\approx 6n-6.$
\end{proposition}

\begin{proof} Similar to the proof of Proposition \ref{lem4} and \ref{th31},
the solution of the ODEs
\eqref{GLE} with (\ref{patI21}) can be given as
 \eqref{solgle21}, where
\begin{equation}\label{p51}
\dot\rho_1(t)+\dot\rho_2(t)=8n-12, \qquad
\dot\rho_2(t)-\dot\rho_1(t)=4n+\frac{
8n\rho_1^{n/2}(t)}
{\rho_2^{n/2}(t)-\rho_1^{n/2}(t)}, \qquad t>0, \nonumber
\end{equation}
which implies
\be\label{p512}
\dot\rho_1(t)=2n-6-\frac{4n\rho_1^{n/2}(t)}
{\rho_2^{n/2}(t)-\rho_1^{n/2}(t)},\qquad \dot\rho_2(t)=6n-6+\frac{4n
\rho_1^{n/2}(t)}
{\rho_2^{n/2}(t)-\rho_1^{n/2}(t)},\qquad t>0.
\ee

1) When $n=2$, \eqref{p512} reduces to
\be\label{p513}
\dot\rho_1(t)=-2-\frac{8\rho_1(t)}
{\rho_2(t)-\rho_1(t)},\qquad \dot\rho_2(t)=6+\frac{8
\rho_1(t)}
{\rho_2(t)-\rho_1(t)},\qquad t>0.
\ee
Solving \eqref{p513} with the initial data \eqref{initrh120}, we get
\be
\rho_1(t)=2t+C_1-\sqrt{8t^2+8C_1t+C_2^2}, \quad
\rho_2(t)=2t+C_1+\sqrt{8t^2+8C_1t+C_2^2}, \quad t\ge0. \nonumber
\ee
Thus there exists a $T_c:=\frac{1}{2}\left[-C_1+\sqrt{C_1^2+2a_1^2a_2^2}\right]>0$,
such that
\be
\rho_1(T_c)=0,\quad \rho_2(T_c)>0,\quad  \rho_1(t)>0,\quad \rho_2(t)>0,\quad t\in[0, T_c), \nonumber
\ee
which immediately implies that  $\mathbf x_1(t)$, $\mathbf x_2(t)$ and $\mathbf x_5(t)$ be a collision cluster among the $5$ vortices and they will
collide at the origin $(0,0)^T$  when $t\to T_c^-$.

 2) When $n=3$, \eqref{p512} reduces to
 \be\label{p515}
\dot\rho_1(t)=-\frac{12\rho_1^{3/2}(t)}{
\rho_2^{3/2}(t)-\rho_1^{3/2}(t)},\qquad
\dot\rho_2(t)=\frac{12\rho_2^{3/2}(t)}{\rho_2^{3/2}(t)-\rho_1^{3/2}(t)},
\qquad t>0,
\ee
which immediately implies
\be \label{p514}
\frac{d}{dt}\left[\frac{1}{\sqrt{\rho_1(t)}}+\frac{1}{\sqrt{\rho_2(t)}}\right]=0
\Longrightarrow \frac{1}{\sqrt{\rho_1(t)}}+\frac{1}{\sqrt{\rho_2(t)}}\equiv
\frac{a_1+a_2}{a_1a_2}, \qquad t\ge0.
\ee
Since $0<a_1<a_2$, then $\rho_1(t)$ and  $\rho_2(t)$ are monotonically decreasing and increasing functions, respectively. From \eqref{p514},
we know that $0<\rho_1(t)<\rho_2(t)$ for $t\ge0$ and thus $T_{\rm max}=+\infty$, i.e. there is no finite time collision. Noting that $M_0>0$, by Theorem
\ref{fi}, we have
\be\label{p517}
\lim_{t\rightarrow+\infty}\rho_2(t)=+\infty.
\ee
Combining \eqref{p517}, \eqref{p514} and \eqref{p515}, we obtain
\eqref{p516} immediately.

3) When $n\geq 4$, the proof is analogue to that
of Proposition \ref{lem4} and thus it is
omitted here for brevity.

In addition,
Figure \ref{rho123} depicts the solution
$\rho_1(t)$ and $\rho_2(t)$ of \eqref{p512} obtained numerically with $\rho_1^0=1$ and $\rho_2^0=4$ for different $n\ge2$.
\end{proof}

\begin{figure}[t!]
\centerline{\psfig{figure=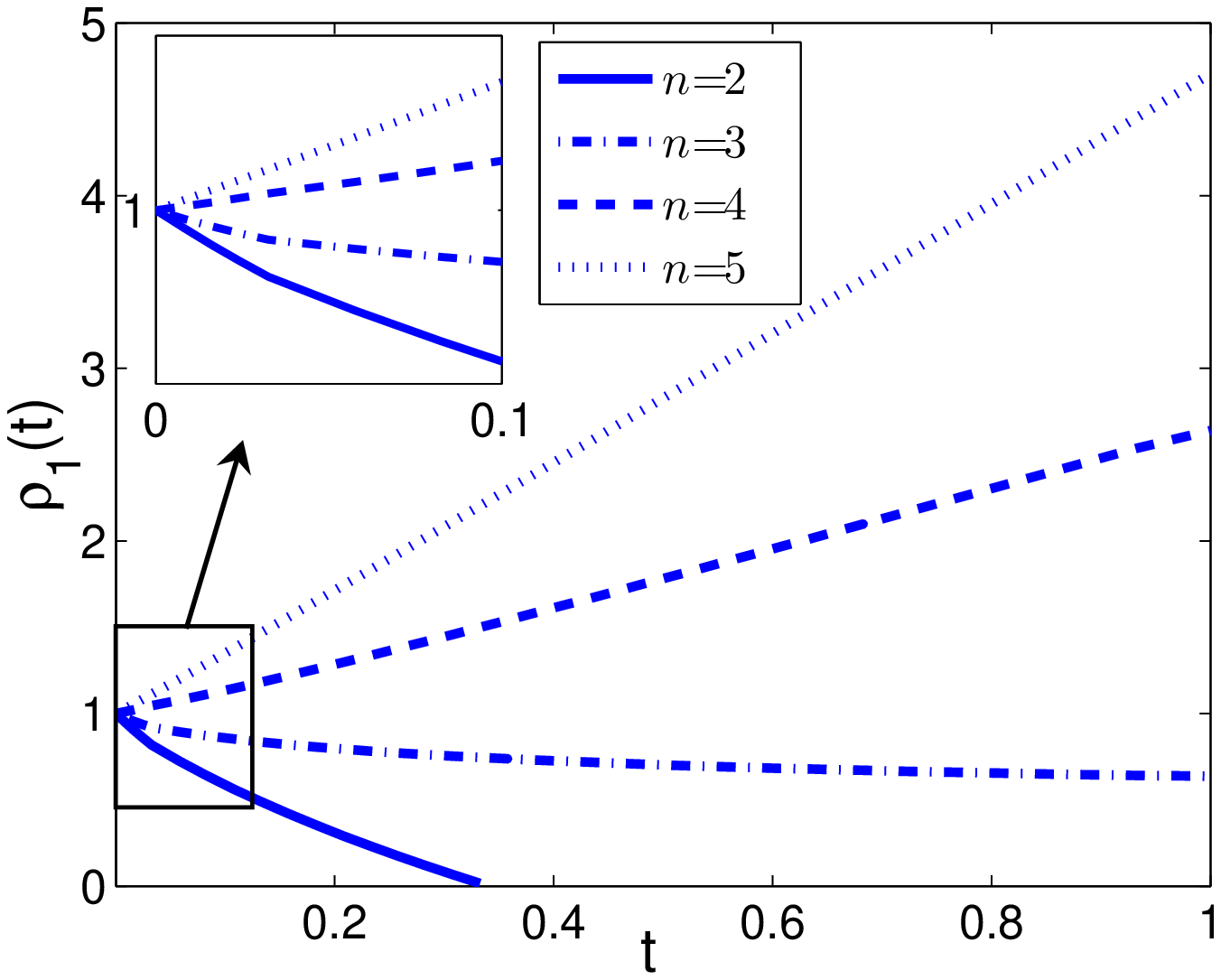,height=5cm,width=6.5cm,angle=0}
\psfig{figure=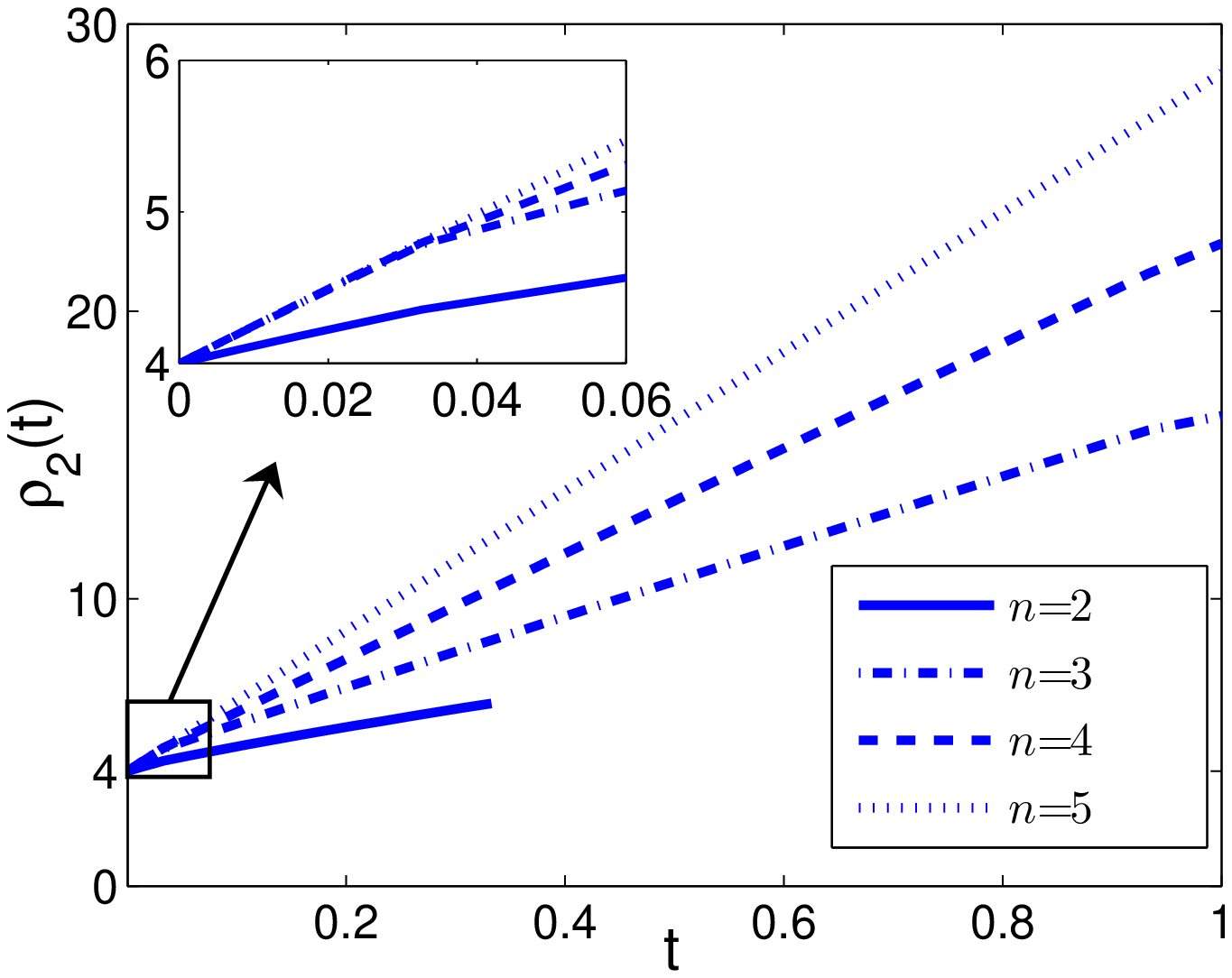,height=5cm,width=6.5cm,angle=0}}
\caption{Time evolution of $\rho_1(t)$ (left) and $\rho_2(t)$ (right)
of \eqref{p512} with $\rho_1^0=1$ and $\rho_2^0=4$ for different $n\ge2$.}
\label{rho123}
\end{figure}

\begin{proposition}\label{lem9}
Taking $m_N=-m_0$, $m_j=m_0$  for $1\le j\le 2n=N-1$ and
the initial data $\bX^0$ in \eqref{GLEi} as
\eqref{patI22}, then the solution of the ODEs
\eqref{GLE} with \eqref{patI22} can be given as
 \eqref{solgle22},
where

(i) when $n=2$, only the three vortices $\mathbf x_1(t),\ \mathbf x_2(t)$ and $\mathbf x_5(t)$  be a collision cluster among the $5$ vortices and they will
collide at the origin $(0,0)^T$ in finite time;

(ii) when $n=3$, then \be
\rho_1(t)\sim\left(\frac{a_1a_2}{a_2-a_1}\right)^2,
\qquad \rho_2(t)\sim 12t, \qquad t\gg1; \nonumber
\ee

(iii) when $n\geq 4$, then
\be
\rho_1(t)\sim \alpha_6 t,\qquad
\rho_2(t)\sim \beta_6 t, \qquad t\gg 1, \nonumber
\ee
with $\alpha_6$ and $\beta_6$ being two positive constants satisfying
\be
0<\alpha_6<\beta_6, \quad \alpha_6+\beta_6=8n-12,\quad
\beta_6-\alpha_6=4n\frac{\beta_6^{n/2}-\alpha_6^{n/2}}
{\beta_6^{n/2}+\alpha_6^{n/2}}. \nonumber
\ee
Specifically, when $n\gg1$, we have $\alpha_6\approx 2n-6$, and $\beta_6\approx 6n-6$.

\end{proposition}

\begin{proof} The proof is analogue to that of Proposition
\ref{lem6} and thus it is
omitted here for brevity.
\end{proof}

\begin{proposition}\label{lem2}
Taking $m_N=-m_0$, $m_j=m_0$ and $m_{n+j}=-m_0$ for $1\le j\le n$ and
the initial data $\bX^0$ in \eqref{GLEi} as
\eqref{patI22}, then the solution of the ODEs
\eqref{GLE} with \eqref{patI22} can be given as
 \eqref{solgle22},
where

(i) when $n=2$, only the three vortices $\mathbf x_1(t),\ \mathbf x_2(t)$ and $\mathbf x_5(t)$  be a collision cluster among the $5$ vortices and they will
collide at the origin $(0,0)^T$ in finite time;

(ii) When $n\geq 3$, all the $N=2n+1$ vortices be a collision cluster and they will collide at the origin $(0,0)^T$ when $t\to T_c:=\frac{1}{4}[a_1^2+a_2^2]$.
\end{proposition}


\begin{proof} Similar to the proof of Proposition \ref{lem4} and \ref{th31},
the solution of the ODEs
\eqref{GLE} with (\ref{patI22}) can be given as
 \eqref{solgle22}, where
\begin{equation}\label{481}
\dot\rho_1(t)+\dot\rho_2(t)=-4, \qquad
\dot\rho_2(t)-\dot\rho_1(t)=8-4n+\frac{
8n\rho_1^{n/2}(t)}{\rho_2^{n/2}(t)+\rho_1^{n/2}(t)}, \qquad t>0,
\end{equation}
which implies
\be\label{482}
\dot\rho_1(t)=2n-6-\frac{4n\rho_1^{n/2}(t)}
{\rho_2^{n/2}(t)+\rho_1^{n/2}(t)},\qquad \dot\rho_2(t)=2-2n+\frac{4n
\rho_1^{n/2}(t)}
{\rho_2^{n/2}(t)+\rho_1^{n/2}(t)},\qquad t>0.
\ee

Solving \eqref{481} with initial data \eqref{initrh120}, we get
\be
\rho_1(t)+\rho_2(t)=-4t+a_1^2+a_2^2, \qquad t\ge0, \nonumber
\ee
which implies that a finite time collision must happen and
$0<T_{\rm max}\le T_c:=\frac{1}{4}(a_1^2+a_2^2)$.

1) When $n=2$, from \eqref{482}, we obtain
\be\label{rh1283}
\dot\rho_1(t)=-\frac{10\rho_1(t)+2\rho_2(t)}{\rho_1(t)+\rho_2(t)}<0,
\qquad \dot\rho_2(t)=\frac{6\rho_1(t)-2\rho_2(t)}{\rho_1(t)+\rho_2(t)},
\qquad t>0.
\ee
Solving \eqref{rh1283} with the initial data \eqref{initrh120}, we have

\begin{equation*}
\rho_1(t)=(-2t+C_1)(2C_3(-2t+C_1)-1) ,\
  \rho_2(t)=(-2t+C_1)(3-2C_3(-2t+C_1)), \ t\ge0.
\end{equation*}
where $C_3=\frac{3a_1^2+a_2^2}{(a_1^2+a_2^2)^2}>0$.
Denote
\be
0<T_{\rm max}=\frac{2C_1C_3-1}{4C_3}=\frac{a_1^2(a_1^2+a_2^2)}
{6a_1^2+2a_2^2}=\frac{2a_1^2}
{3a_1^2+a_2^2}T_c<T_c, \nonumber
\ee then we have
\be
\rho_1(T_{\rm max})=0,\quad \rho_2(T_{\rm max})>0,\qquad   \rho_1(t)>0,\quad \rho_2(t)>0,
\quad t\in [0,T_{\rm max}), \nonumber
\ee
which implies that only the three vortices $\mathbf x_1(t),\ \mathbf x_2(t)$ and $\mathbf x_5(t)$  be a collision cluster among the $5$ vortices
and they will
collide at the origin $(0,0)^T$ when $t\to T_{\rm max}^-$.

2) When $n\geq 3$, by Theorem \ref{fi}, only the $n+1$ vortices
with $\mathbf x_{n+1}(t),\ldots,\mathbf x_{2n}(t)$ and $\mathbf x_N(t)$  cannot be a collision cluster among
the $N$ vortices since they have the same winding number; and similarly, only the $n+1$ vortices
with $\mathbf x_1(t),\ldots,\mathbf x_n(t)$ and $\mathbf x_N(t)$
 cannot be a collision cluster among the $N$ vortices since
their collective winding number defined as $M_1:=\sum_{1\le j<l\le n}m_jm_l+\sum_{j=1}^n m_jm_N=\frac{1}{2}[(n-1)^2-n-1]=\frac{1}{2}n(n-3)\ge0$.
Thus, in order to have a finite time collision,
there exist $1\leq j_0\leq n$ and $1\leq l_0\leq
n$ such that the vortex dipole $\mathbf x_{j_0}(t)$ and $\mathbf
x_{n+l_0}(t)$ will collide at $t=T_c$, i.e. $\rho_1(T_c)=\rho_2(T_c)=0$.
Therefore, the $N=2n+1$ vortices will be a (finite time) collision cluster.

In addition,
Figure \ref{rho126} depicts the solution
$\rho_1(t)$ and $\rho_2(t)$ of \eqref{482} obtained numerically with $\rho_1(0)=1$ and $\rho_2(0)=4$ for different $n\ge2$.
\end{proof}

\begin{figure}[t!]
\centerline{\psfig{figure=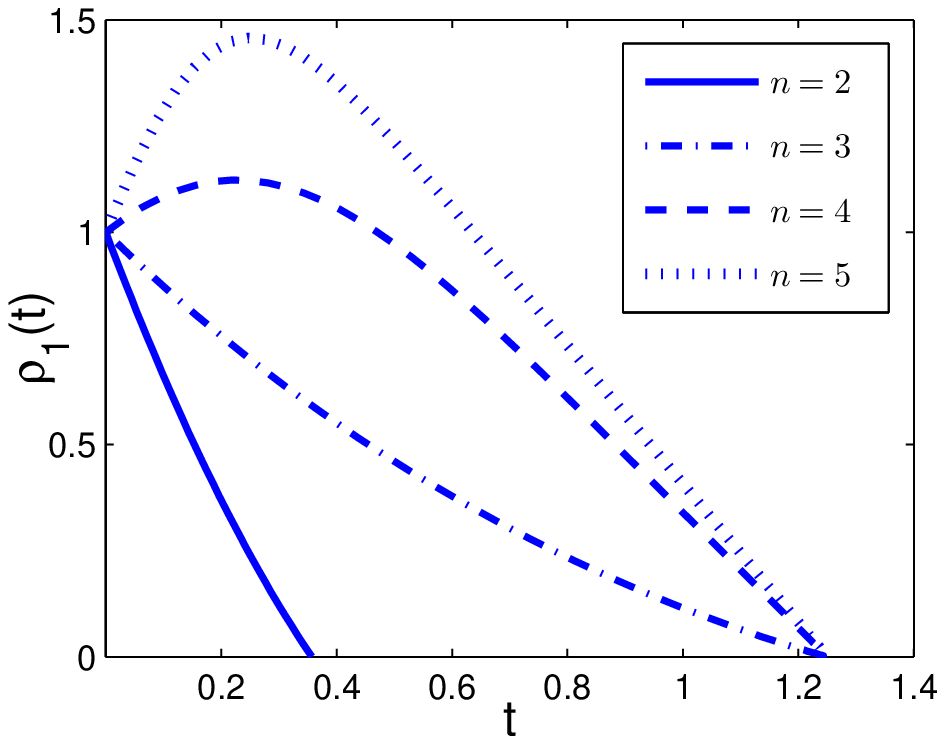,height=5cm,width=6.5cm,angle=0}
\psfig{figure=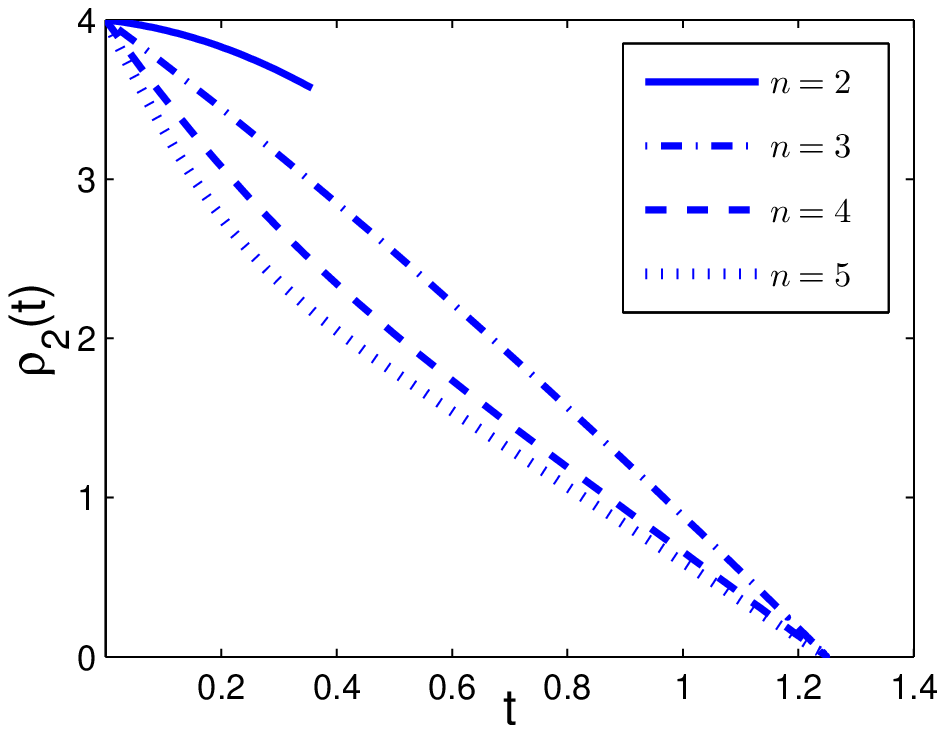,height=5cm,width=6.5cm,angle=0}}
\caption{Time evolution of $\rho_1(t)$ (left) and $\rho_2(t)$ (right)
of \eqref{482} with $\rho_1^0=1$ and $\rho_2^0=4$ for different $n\ge2$.}
\label{rho126}
\end{figure}

\section{Conclusion}
Based on the reduced dynamical law of a system of ordinary differential
equations (ODEs) for the dynamics of $N$ vortex centers, we
have obtained stability and interaction patterns
of quantized vortices in superconductivity.
By deriving several non-autonomous first integrals of the
ODEs system, we proved  global well-posedness
of the $N$ vortices when they have the same winding number and
demonstrated that finite time collision might happen when they
have different winding numbers.
When $N=3$, we established rigorously
orbital stability when they have the same winding number
and classified their collision patterns when they have different
winding numbers. Finally, under several special initial setups
 including interaction of
two clusters, we obtained explicitly the analytical solutions
of the ODEs system. The analytical and numerical results demonstrated
the rich dynamics and interaction patterns of $N$ vortices in
superconductivity.

\section*{Acknowledgments}
This work was supported partially by the Academic Research
Fund of Ministry of Education of Singapore grant No.
R-146-000-223-112 (W.B.) and
by the National Natural Science Foundation of China grant No. 11371166,  11501242 (S.S. and Z.X.).




\begin{thebibliography}{99}

\bibitem{Bao}
\newblock W. Bao,
\newblock \emph{Numerical methods for the nonlinear Schr\"odinger equation with nonzero far-field conditions},
\newblock Methods Appl. Anal., \textbf {11} (2004), 367-387.

\bibitem{Bao0}
\newblock W. Bao and Y. Cai,
\newblock \emph{Mathematical theory and numerical methods for Bose-Einstein condensation},
\newblock Kinet. Relat. Mod., \textbf{6} (2013), 1-135.

\bibitem{Bao1}
\newblock W. Bao and Q. Tang,
\newblock \emph{Numerical study of quantized vortex interaction in the nonlinear Schroedinger equation on bounded domains},
\newblock Multiscale Model. Simul.,
\textbf{12} (2014),  411-439.

\bibitem{Bao2}
\newblock W. Bao and Q. Tang,
\newblock \emph{Numerical study of quantized vortex interaction in the Ginzburg-Landau equation on bounded domains},
\newblock Commun. Comput. Phys., \textbf{14} (2013), 819-850.

\bibitem{Bao3}
\newblock W. Bao,  R. Zeng and Y. Zhang,
\newblock \emph{Quantized vortex stability and interaction in the nonlinear wave equation},
\newblock Phys. D, \textbf{237} (2008), 2391-2410.

\bibitem{Bauman}
\newblock P. Bauman, C. Chen, D. Phillips and P. Sternberg,
\newblock \emph{Vortex annihilation in nonlinear heat flow for Ginzburg-Landau systems},
\newblock European J. Appl. Math., \textbf {6} (1995), 115-126.

\bibitem{Bethuel}
\newblock F. Bethuel, H. Brezis and F. H\'elein,
\newblock \doititle{``Ginzburg-Landau Vortices"},
\newblock Birkh\"auser, Boston, 1994.

\bibitem{Chapman}
\newblock S. J. Chapman and G. Richardson,
\newblock \emph{Motion of vortices in type II superconductors},
\newblock SIAM J. Appl. Math., \textbf {55} (1995), 1275-1296.

\bibitem{Colliander}
\newblock J. E. Colliander and R. L. Jerrard,
\newblock \emph{Vortex dynamics for the Ginzburg-Landau-Schr\"odinger equation},
\newblock Internat. Math. Res. Notices, \textbf {7} (1998), 333-358.

\bibitem{Du}
\newblock Q. Du,
\newblock \emph{Finite element methods for the time-dependent Ginzburg-Landau model of superconductivity},
\newblock Comput. Math. Appl., \textbf{27} (1994), 119-133.

\bibitem{E}
\newblock W. E,
\newblock \emph{Dynamics of vortices in Ginzburg-Landau theories with applications to superconductivity},
\newblock Phys. D, \textbf{77} (1994), 383-404.

\bibitem{Jerrard}
\newblock R. Jerrard and H. M. Soner,
\newblock \emph{Dynamics of Ginzburg-Landau vortices},
\newblock Arch. Rat. Mech., \textbf {142} (1998), 99-125.

\bibitem{Klein}
\newblock A. Klein, D. Jaksch, Y. Zhang and W. Bao,
\newblock \emph{Dynamics of vortices in weakly interacting Bose-Einstein condensates},
\newblock Phys. Rev. A,  \textbf{76} (2007), 043602.

\bibitem{Lange}
\newblock O. Lange and B. Schroers,
\newblock \emph{Unstable manifolds and Schr\"odinger dynamics of Ginzburg-Landau vortices},
\newblock Nonlinearity, \textbf{15} (2002), 1471-1488.

\bibitem{Lin96}
\newblock F. Lin,
\newblock \emph{Some dynamical properties of Ginzburg-Landau vortices},
\newblock  Comm. Pure Appl. Math., \textbf{49} (1996), 323-360.

\bibitem{Lin98}
\newblock F. Lin,
\newblock \emph{Complex Ginzburg-Landau equations and dynamics of vortices, filaments, and codimension-2 submanifolds},
\newblock Comm. Pure Appl. Math., \textbf{51} (1998), 385-441.

\bibitem{Lin99}
\newblock F. Lin and J. Xin,
\newblock \emph{On the dynamical law of the Ginzburg-Landau vortices on the plane},
\newblock  Comm. Pure Appl. Math., \textbf{52} (1999), 1189-1212.

\bibitem{Mironescu}
\newblock P. Mironescu,
\newblock \emph{On the stability of radial solutions of the Ginzburg-Landau equation},
\newblock J. Funct. Anal., \textbf{130} (1995), 334-344.

\bibitem{Newton}
\newblock P. K. Newton and G. Chamoun,
\newblock \emph{Vortex lattice theory: a particle interaction perspective},
\newblock SIAM Rev., \textbf{51} (2009), 501-542.

\bibitem{Neu1}
\newblock J. Neu,
\newblock \emph{Vortices in complex scalar fields},
\newblock Phys. D, \textbf{43} (1990), 385-406.

\bibitem{Neu2}
\newblock J. Neu,
\newblock \emph{Vortex dynamics of the nonlinear wave equation},
\newblock Phys. D, \textbf{43} (1990), 407-420.

\bibitem{Ovchinnikov1}
\newblock Y. Ovchinnikov and I. Sigal,
\newblock \emph{Long-time behavior of Ginzburg-Landau vortices},
\newblock Nonlinearity, \textbf{11} (1998), 1295-1309.

\bibitem{Ovchinnikov2}
\newblock Y. Ovchinnikov and I. Sigal,
\newblock \emph{Asymptotic behavior of solutions of Ginzburg-Landau and relate equations},
\newblock Rev. Math. Phys., \textbf{12} (2000), 287-299.

\bibitem{Pit}
\newblock L. P. Pitaevskii and S. Stringari,
\newblock \emph{Bose-Einstein Condensation},
\newblock Clarendon Press, Oxford, 2003.


\bibitem{Sandier}
\newblock E. Sandier,
\newblock \emph{The symmetry of minimizing harmonic maps from a two-dimernsional domain to the sphere},
\newblock Ann. Inst. H. Poincar\'e Anal. Non Lin\'eaire, \textbf{10} (1993), 549-559.

\bibitem{Zhang1}
\newblock Y. Zhang, W. Bao and Q. Du,
\newblock \emph{The dynamics and interaction of quantized vortices in the Ginzburg-Landau-Schr\"odinger equation},
\newblock SIAM J. Appl. Math., \textbf{67} (2007), 1740-1775.

\bibitem{Zhang2}
\newblock Y. Zhang, W. Bao, and Q. Du,
\newblock \emph{Numerical simulation of vortex dynamics in Ginzburg-Landau-Schr\"odinger equation},
\newblock European J. Appl. Math., \textbf{18} (2007), 607-630.






\end{thebibliography}
\end{document}